\documentclass[10pt]{amsart}
\usepackage[utf8]{inputenc}
\usepackage[T1]{fontenc}
\usepackage[square,sort,comma,numbers]{natbib}
\usepackage{graphicx}
\usepackage{amsmath}
\usepackage{amssymb}
\usepackage{bm}
\usepackage{amsfonts}
\usepackage{mathtools}
\usepackage{enumitem}
\usepackage{tikz}
\usepackage[a4paper, margin=1in, footskip=0.25in]{geometry}
\usepackage{amsthm}
\usepackage{xcolor}
\usepackage{mathrsfs}
\usepackage[hidelinks]{hyperref}
\newtheorem{theorem}{Theorem}[section]
\newtheorem{corollary}{Corollary}[theorem]
\newtheorem{lemma}[theorem]{Lemma}
\newtheorem{proposition}{Proposition}[section]

\theoremstyle{definition}
\newtheorem{definition}{Definition}[section]
\newtheorem{assumption}{Assumptions}[section]

\newtheorem{remark}{Remark}[section]

\newcommand{\bluepassage}[1]{{#1}}
\newcommand{\redpassage}[1]{}

\newcommand{\NN}{\mathbb{N}}

\newcommand{\EE}{\mathbb{E}}
\newcommand{\PP}{\mathbb{P}}
\newcommand{\PC}{\mathcal{P}}
\newcommand{\FF}{\mathcal{F}}
\newcommand{\QQ}{\mathbb{Q}}

\newcommand{\AAC}{\mathcal{A}}

\newcommand{\nois}{B}

\newcommand{\RR}{\mathbb{R}}

\newcommand{\absv}[1]{\left| #1 \right|}

\newcommand{\mean}[1]{\langle #1 \rangle}

\newcommand{\inv}[1]{\frac{1}{#1}}

\newcommand{\norm}[1]{\left\Vert #1 \right\Vert}

\newcommand{\varst}[2]{\norm{#1}_{#2,[s,t]}}

\newcommand{\indic}[1]{\mathbf{1}_{[#1]}}

\newcommand{\qvarst}[1]{\varst{#1}{q}}

\newcommand{\qvart}[1]{\norm{#1}_{q, [0,t]}}

\newcommand{\wh}{W^H}

\newcommand{\ls}{\lesssim}

\newcommand{\eps}{\varepsilon}
\newcommand{\dqivar}[1]{\norm{ \delta_i #1 }_q}
\newcommand{\dqvar}[2]{\norm{ \delta_{#2} #1 }_q}

\newcommand{\ccc}[1]{\mathcal{C}^{#1}}
\newcommand{\vvv}[1]{\ccc{#1-var}}
\newcommand{\ccdv}{\vvv{1/\delta}}

\newcommand{\absvv}[1]{\norm{#1}_{L^{\infty}}}

\newcommand{\fsp}{L_t^p\ccc{\alpha}_x}
\newcommand{\nofn}{\norm{ f}_{L_t^p\ccc{\alpha}_x}}

\newcommand{\ccax}{\ccc{\alpha}_x}

\renewcommand{\qvarst}[1]{\varst{#1}{q-var}}
\newcommand{\qvarstCx}[1]{\varst{#1}{q-var(\ccc{1}_x)}}

\renewcommand{\qvart}[1]{\norm{#1}_{q-var, [0,t]}}
\newcommand{\qvartCx}[1]{\norm{#1}_{q-var(\ccc{1}_x),[0,t]}}
\newcommand{\cH}{\mathcal{H}^{H},T}

\begin{document}

\title{Perturbations of singular fractional SDEs}

\author{Paul Gassiat}
\address{ Universit\'e Paris-Dauphine, PSL University, UMR 7534, CNRS, CEREMADE, 75016 Paris, France}
\email{gassiat@ceremade.dauphine.fr}

\author{Łukasz Mądry}

\address{ Universit\'e Paris-Dauphine, PSL University, UMR 7534, CNRS, CEREMADE, 75016 Paris, France}
\email{madry@ceremade.dauphine.fr}

\maketitle {\noindent
\small

}

\begin{abstract}
We obtain well-posedness results for a class of ODE with a singular drift and additive fractional noise, whose right-hand-side involves some bounded variation terms depending on the solution. Examples of such equations are reflected equations, where the solution is constrained to remain in a rectangular domain, as well as so-called perturbed equations, where the dynamics depend on the running extrema of the solution. Our proof is based on combining the Catellier-Gubinelli approach based on Young nonlinear integration, with some Lipschitz estimates in $p$-variation for maps of Skorokhod type, due to Falkowski and Słominski. An important step requires to prove that fractional Brownian motion, when perturbed by sufficiently regular paths (in the sense of $p$-variation), retains its regularization properties. This is done by applying a variant of the stochastic sewing lemma.
\end{abstract}

\section{Introduction}

Regularization by noise of ODEs is by now a thoroughly studied subject. The main observation is that, even though classical ODE theory requires $f=f(t,x)$ to be Lipschitz w.r.t. $x$ for well-posedness of $\dot{x}(t) = f(t,x)$ to hold, one can decrease this regularity requirement by adding an (irregular) noise term to the equation, namely by considering equations of the form
\begin{equation} \label{eq:SDEintro}
dX_t = f(t,X_t) dt +  dW_t. 
\end{equation}
Indeed, in the case where $W$ is a Brownian motion, well-posedness of the above only requires $f$ to be bounded, as noted in the classical works of Zvonkin and Veretennikov \cite{Zvo74,Ver81}. More recently, in the seminal work \cite{CG16}, Catellier and Gubinelli have observed that the results could be extended to fractional Brownian paths with Hurst index $H \in (0,1)$ (extending to arbitrary dimension results obtained by Nualart-Ouknine \cite{NO02} in the scalar case). An important feature of their results is that the vector field $f$ can be chosen arbitrarily irregular, provided that $H$ is small enough (namely, that the added noise is sufficiently irregular).
\smallskip

Following \cite{CG16}, regularization by noise for fractional SDE has been studied in recent years by many authors and in various contexts, see for instance \cite{GG20, HP20, GHM21, CDR22}. It has also been remarked that the so-called stochastic sewing lemma, which was formalized by L\^e \cite{Le20}, is a very useful technical tool to obtain this kind of result, see for instance \cite{BDG19,Ger22,ABLM21, ART21}.
\smallskip

In the present work, we extend the Catellier-Gubinelli framework to deal with ODEs incorporating some path-dependent terms which are sufficiently regular (in the $q$-variation scale). For concreteness, we present here two families of equations for which we obtain well-posedness (in both cases the perturbations are of bounded variation).

The first one concerns \emph{reflected equations}, namely equations where the solution is constrained to take value in $D = \Pi_{i=1}^d [a_i, b_i]$ for some $a_i < b_i \in [-\infty, \infty]$. These can be written as
\begin{equation} \label{eq:Rsde1}
X_t = x_0 + \int_0^t f(s,X_s) ds + W_t + K_t, 
\end{equation}
where the unknown is the pair $(X,K)$ which must satisfy the additional constraint
\begin{equation} \label{eq:Rsde2}
 \forall t \geq 0, X_t \in D, \;\;\;\mbox{ and for }i=1, \ldots, d, \;\; dK^i_t =dK^{+,i}_t  1_{\{X^i_t = a_i\}} - dK^{-,i}_t  1_{\{X^i_t = b_i\}} \end{equation}
where the $K^{\pm,i}$ are non-decreasing.

The second example of equations concerns so-called \emph{perturbed equations}, which can now be written component-wise as
\begin{equation} \label{eq:Pert}
X_t^i = x_0^i + \int_0^t f^i(s,X_s) ds + W_t^i +  \alpha_i \max_{0 \leq s \leq t} X^i_s +  \beta_i \min_{0 \leq s \leq t} X^i_s.
\end{equation}

Here the constants $\alpha_i$, $\beta_i$ satisfy
\begin{equation*}
\forall i=1,\ldots, d, \;\; \alpha_i < 1, \; \beta_i < 1, \;\; \frac{| \alpha_i \beta_i|}{(1-\alpha_i)(1- \beta_i)} < 1.
\end{equation*}

Our main result is then the following theorem (we refer to Theorem \ref{thm:main} below for a more precise statement which in particular allows for $t$-dependent $f$).

\begin{theorem} \label{thm:intro}
Let $W$ be a fractional Brownian motion of Hurst index $H \in (0,1)$. Assume that $ f=f(x)\in C^\alpha$, where
\[\alpha > \left( 1 - \frac{1}{2H} \right) \vee \left( 2 - \frac{1}{H} \right). \]
Then path-by-path uniqueness, strong existence and stability estimates hold for both reflected equations \eqref{eq:Rsde1}-\eqref{eq:Rsde2} and perturbed equations \eqref{eq:Pert}.
\end{theorem}

We remark that when $H< \frac{1}{2}$, the regularity requirement on $f$ coincides with that obtained in \cite{CG16}, which is believed to be optimal. However when $H>\frac 1 2$, we require more regularity, and we do not believe our statement to be optimal (but this still allows for some non-Lipschitz vector fields, for any $H<1$).

\bluepassage{Path-by-path uniqueness (which is precisely defined in Definition \ref{def:solution} below) means uniqueness for (almost every) fixed $\omega$, among all solutions for which the equation makes sense, not necessarily those given by adapted processes as in the more classical (but typically weaker) notion of strong (or pathwise) uniqueness. This type of uniqueness was first obtained in \cite{Dav07} in a regularization by noise context, see also the discussion in Ch. 1 of \cite{Fla11}, and the recent work \cite{SW22} for an example where uniqueness holds in the strong sense but not path-by-path.
}

\smallskip

Let us now explain the idea of the proof of Theorem \ref{thm:intro}. The key point is that, even though the additional terms on the r.h.s. of \eqref{eq:Rsde1} and \eqref{eq:Pert} can not be expressed as functions (or even Schwartz distributions) of $X_t$ as in \eqref{eq:SDEintro}, we can nevertheless recast both the reflected and perturbed equations as fixed points
 \begin{equation} \label{eq:FPintro}
 X = \Gamma \left( x_0 + \int_0^{\cdot} f(s,X_s) ds + W \right), 
 \end{equation}
and the map $\Gamma$ (which in the case of reflected equations is the well-known Skorokhod map) is in fact Lipschitz continuous in any $p$-variation spaces, as follows from results of Falkowski and Słominski \cite{FS15,FS22}. This turns out to be pleasantly compatible with the approach of Catellier-Gubinelli \cite{CG16}, in which $\int_0^t f(s,X_s) ds$ is interpreted as a (nonlinear) Young integral. 

\smallskip

Recall that the main step in obtaining path-by-path uniqueness and stability, following respectively \cite{CG16,GHM21}, is a linearization procedure around a given solution $X$, which has to be combined with the observation that
 \[ T^X f(t,x) := \int_0^t f(s, X_s +x) ds \]
 is typically much more regular than $f$ itself. To obtain that fact, \cite{CG16} (inspired by \cite{Dav07}) relied on an application of Girsanov's theorem, by which it was enough to obtain this regularity for the path $W$ itself. However, in our setting, the solution $X$ will be, after a Girsanov transformation, equivalent not to $W$ but to $\Gamma(W)$, so that we need to check that the latter still has the same regularizing properties as the former (even though they are not equivalent via a Girsanov transformation). This is achieved by noting that $\Gamma(W) = W + C(W)$, where $C(W)$ is an adapted process of bounded variation. We then obtain, using a stochastic sewing procedure, that this fact is sufficient to obtain the same regularization effect for $\Gamma(W)$ as that of $W$ (at least when $H< 1/2$). 
\smallskip

\smallskip

We conclude this introduction by a few words on reflected and perturbed equations. Reflected diffusions, historically first defined by Skorokhod \cite{Sko61}, arise naturally when considering processes constrained to remain in a subdomain of $\RR^d$. The case of classical Brownian SDE has been studied by numerous authors, starting with Tanaka \cite{Tan79} and Lions-Sznitman \cite{LS84}. In the case of singular drifts as considered here, well-posedness for bounded $b$ was obtained in \cite{Zha94} in the scalar case and only recently in \cite{YZ20} in the multi-dimensional case. Concerning reflected equations driven by fractional Brownian motions with Hurst index $H$, for general multiplicative noise, well-posedness has been obtained in the case $H>\frac{1}{2}$ by Falkowski and Słominski \cite{FS15,FS22}, whereas for $H< \frac{1}{2}$ the situation is more subtle, since existence holds \cite{Aid15} but uniqueness only holds in the scalar case \cite{DGHT19} and not in higher dimension \cite{Gas21}. In the case of additive noise considered in the present paper however, there are no such difficulties and well-posedness holds for any Lipschitz $f$ and arbitrary path $W$ (see e.g. \cite{CMR22}).  

Perturbed equations of the form \eqref{eq:Pert} allow naturally to give simple and tractable examples of continuous processes which interact with their past,  and as such have been studied by several authors, initially in the case of perturbed Brownian motions (i.e. $f \equiv 0$, $H=1/2$). The condition on $\alpha_i$, $\beta_i$ above was introduced by \cite{CPY98}, who showed that this allowed to set up the equation as a contraction on path-space, approach that we follow here. Sharpness of this condition has been studied by \cite{Dav96} , while \cite{CD99} show that $\alpha, \beta<1$ is sufficient for strong uniqueness when dealing with Brownian paths. Perturbed SDE (driven by classical Brownian motion) have then been considered (sometimes in combination with reflection) by a number of authors, see e.g. \cite{DZ05, YZ15}, and in particular \cite{BHO09} which obtain some well-posedness results for equations with singular drits (in the scalar case).

\vspace{4mm}

The rest of the article is organized as follows. In Section \ref{sec:Prem}, we state preliminary facts about fractional Brownian motion (fBm) and (nonlinear) Young integration. In Section \ref{sec:avgpert}, we obtain some regularity estimates on averaged fields for fBm perturbed by an adapted path of finite $q$-variation. In Section \ref{sec:FixedPointEqs}, we state and prove our main results on fixed point equations of the form \eqref{eq:FPintro}. In Section \ref{sec:RefPert}, we check that both reflected and perturbed equations fall in the framework of the previous section. Finally, some of the arguments and computations are relegated to an Appendix. 

\section{Preliminaries} \label{sec:Prem}

\subsection{Notations}

Throughout the paper, the time horizon $T>0$ will be fixed.

For a $V$-valued path $X \in \mathcal{C}([0,T],V)$ we define $X_{st} = X_t - X_s$. Let $\Delta_2$ be the 2-dimensional simplex, that is $\Delta_2 = \{ (s,t) : 0 \leq s \leq t \leq T \}$. Then for $X: \Delta_2 \mapsto V$, denoted for $(s,t) \in \Delta_2$ as $X_{st}$ we define $\delta$ as $\delta_u X_{st} = X_{st} - X_{su} - X_{ut}$ for $s < u < t$. In the second case, that is where $X_{st}$ is to be understood as a two-parameters map rather than an increment of a path, we will frequently call these objects "germs". If $\delta_u X_{st} = 0$, then $X$ is a path.

We will use the standard variation and H\"older spaces, with seminorms defined as follows, where $\sup_{\PC}$ means that supremum runs over possible partitions of $[0,T]$:

$$\qvart{X } = \left( \sup_{\PC} \sum_{[u,v] \in \PC} \absv{ X_{uv} }^q \right)^{1/q} < \infty \quad \norm{ X}_{\alpha;[0,t]} = \sup_{s \neq t} \frac{ \absv{ X_{st} }}{{\absv{t-s}^{\alpha}} } $$
 In general, we will denote variation norm using Latin letters and H\"older norms with Greek ones. We will denote variation spaces by $\vvv{q}$ and H\"older spaces by $\ccc{\alpha}$. In particular, the H\"older constant of a $V$-valued path $X$ is denoted by $\norm{ X }_{\ccc{\gamma}_tV_x}$.


Given a $V$-valued random variable on a probability space equipped with a filtration $(\FF_t)_{t\geq0}$, we let
$$\norm{ X }_{\EE;m;V} = \big( \EE \absv{ X }^m_V \big)^{1/m} \quad \EE_s X = \EE X \lvert \FF_s$$

\bluepassage{The notation $\EE_s$ will always mean conditional expectation with respect to the filtration of the underlying Brownian motion $\nois$ defined in subsection \ref{subsect:fbm} below- if there arises a need to specify different filtration $\mathcal{G}_s$, we will use $\EE_{\mathcal{G}_s}$. } By saying that a map $w: \Delta_2(s,t) \mapsto \RR_+$ is a control or a control function we mean that for all $s < u < t$ there is \[ w(s,u) + w(u,t) \leq w(s,t) \quad w(s,t) \to 0\ \text{for}\ t \to s. \] Sometimes we can say that this is a superadditive map, or a superadditive control. A random control is a map for which its superadditivity property holds for almost every $\omega \in \Omega$
\\ 
 We will denote $\ccc{\alpha} = B^{\alpha}_{\infty,\infty}$ the Besov-H\"older spaces for $\alpha \in \RR \setminus \NN$, see \cite{BCD11}, Definition 2.68. \bluepassage{Additionally, for $k \geq 1$, possibly $k = \infty$ the notation $f \in \ccc{k,b}$ will mean that all derivatives up to $k$-th one are bounded}

If $E$ is a function space then by $f \in \ccc{\gamma}_tE_x$ we will mean that $f \in \ccc{\gamma}([0,T]; E)$, \bluepassage{with the } subscript $E_x$ underlining the fact of spatial dependence. We use analogous notation for $f \in \vvv{q}_tE_x$, \bluepassage{ and we will denote the corresponding seminorm by $\norm{f}_{q-var(E_x),[0,t]}$}. If additionally $f$ is random and $\EE \left[ \norm{ f}_{\ccc{\gamma}_tE_x} \right]^k < \infty$ for some $k \geq 1$, we will say $f \in L^k\ccc{\gamma}_tE_x$, same for $L^k\vvv{q}_tE_x$.

Moreover, we will denote with $\mathcal{S}(\RR^d)$ and $\mathcal{S}'(\RR^d)$ respectively spaces of Schwarz functions and tempered distributions.

Finally, we will write $a \lesssim b$ when $a \leq C b$ for some constant $C$.

\subsection{Fractional Brownian motion}\label{subsect:fbm}


A fractional Brownian motion of Hurst index $H \in (0,1)$, is a centered Gaussian process $(\wh_t)_{t\geq0}$ with $\wh_0 = 0$ such that\redpassage{, for some constant $c_H> 0$}:

\begin{equation}\label{eq:fbm}
\forall\;s<t,\;\;\EE \absv{ \wh_t - \wh_s }^2 = \absv{t-s}^{2H}
\end{equation}

In fact, it can be obtained from a standard Brownian motion $\bar{\nois}$ by the formula
\begin{equation}\label{eq:rlfbm}
\wh_t = \int_0^t K^H(t,s) d\bar{\nois}_s,
\end{equation}
for some suitable kernel $K^H$ (which we will not need to specify). Moreover, $\wh$ and $\bar{\nois}$ can be chosen such that they generate the same filtration $(\bar{\FF}_t)_{t\geq0}$, see \cite{Nua06intro}, page 280.

The Cameron-Martin space of $\wh$ on $[0,T]$ is then defined by:
\begin{equation}\label{eq:cameronmartin}
\mathcal{H}^{H,T} = \{ h: h(t) = \int_0^t  K^H(t,s) f(s) ds, \mbox{ for some } f \in L^2([0,T]) \}
\end{equation}
equipped with the norm $\norm{ h }_{\mathcal{H}^{H,T}} =  \norm{ f }_{L^2}$.
%
%

\begin{theorem}\label{thm:girsanov} Let $\wh$ be an fBm under $\PP$, with $(\bar{\FF}_t)_{0\leq t\leq T}$ its natural filtration. Let $h$ be an $\bar{\FF}$- adapted process such that $\EE \exp( \lambda \norm{ h }_{\mathcal{H},T}^2 ) \leq K(\lambda)< \infty$ for all $\lambda > 0$, and a certain function $K$. Then there exists a measure $\QQ \sim \PP$ under which $\wh - h$ is a fBm . For the changes of measure $\frac{d\PP}{d\QQ}, \frac{d\QQ}{d\PP}$ we furthermore have:

\[ \EE \left[ \left(\frac{d\PP}{d\QQ}\right)^n + \left( \frac{d\QQ}{d\PP} \right)^n \right]^{1/n} \leq C_n < \infty\;\; \forall\ n \in \NN \]
where $C_n$ only depends on $h$ through the function $K$.
\end{theorem}

\begin{proof}
Exponential integrability of Cameron-Martin norm of $h$ means that Novikov's condition is satisfied. Therefore Theorem 2 in \cite{NO02} holds, from which follows the existence of a suitable $\QQ$. For the second part, we proceed in the exact same way as in Theorem 14 in \cite{GG20}.
\end{proof}

We will need to compare the $\mathcal{H}^{H,T}$ norm with some fractional Sobolev (Besov) spaces. Given $q \in [1,\infty)$ and $\delta \in \left(\inv{q},1\right)$ the space $W^{\delta,q}([0,T], \RR^d)$ is the set of all paths $x \in \mathcal{C}([0,T], \RR^d)$ for which:

\begin{equation}\label{eq:fsobol}
\norm{x}_{W^{\delta,q};[0,T]}^q = \int_{[0,T]^2} \left( \frac{ \absv{x_{uv}} }{\absv{v-u}^{\delta+1/q}} \right)^q du dv < \infty,
\end{equation}
%
%


The following proposition gives a useful embedding of the Cameron-Martin space.

\begin{proposition}\label{prop:besovembedding}
Fix $H\in (0,1)$. For any $\kappa, \delta>0$ and $\kappa' > \kappa$ it holds that

\begin{equation}\label{eq:cmbesovembedding}
\ccc{H+1/2+\kappa'} \subset W^{H+1/2+\kappa,2} \subset \mathcal{H}^{H,T} \subset W^{H+1/2-\delta,2}
\end{equation}

In addition, for any $q \geq \inv{H+1/2}$ if $H<1/2$ and any $q \geq 1$ for $H > 1/2$, it holds that
\begin{equation}\label{eq:jainmonrad}
\mathcal{H}^{H,T} \subset \vvv{q}
\end{equation}

\end{proposition}

\begin{proof}

In \eqref{eq:cmbesovembedding} the first inclusion follows from \eqref{eq:fsobol}, the rest by observing that the inverse operator $K^{-1}_H$ corresponds to $D^{H+1/2}$, i.e. $\norm{ K^{-1}_H f }_{L^2} \sim \norm{ D^{H+1/2} f}_{L^2}$ (see \cite{Pic10}, section 5) and combining it with Proposition 5 in \cite{Dec05} and Proposition 5.3.2 from \cite{Zah98} cited therein.

The embedding \eqref{eq:jainmonrad} is shown in \cite{FGGR16}, Examples 2.8 and 2.4.
\end{proof}

In addition to \eqref{eq:rlfbm}, it will also be convenient to work with an other represenation of $\wh$ in terms of a two-sided Brownian motion $({B}_t)_{t\in \RR}$ (which holds on a possibly enlarged probability space) :

\begin{equation}\label{eq:mandelbrot}
\wh_t = \int_{-\infty}^t k(t,s) dB_s, \quad  k(t,r) = c_H\left( (t-r)_+^{H-1/2} - (-r)_+^{H-1/2} \right)
\end{equation}

where $\nois$ is a two-sided Wiener process with filtration $\FF_t = \sigma(\nois_s; s<t)$ \bluepassage{and $c_H = \Gamma(H+1/2)^{-1}$ is a normalising constant} . Then it can be shown (see for instance \cite{GG20}, section 3.2 or \cite{Pic10}, \cite{Nua06intro}) that for fixed $s<t$ we can decompose $\wh_t$ in two parts, the first one of which is independent of $\FF_s$, the other one is $\FF_s$-measurable, $\wh_t = W^{H,1}_{st} + W^{H,2}_{st}$, where both are Gaussian processes with mean zero. 

\bluepassage{We will denote the convolution with a Gaussian density with variance $t$ as $P_t f = (P_t * f)$ }. Then for $f$ bounded

\[ \EE_{\FF_s} f(\wh_r+ x)  =  \EE_{\FF_s} f(\EE_s \wh_r  + x) = (P_{\sigma^2(s,r)} f)(\EE_s \wh_r+x) \]

where $\sigma^2(s,t) := \inv{d} \EE \left( \wh_t - \EE^{\FF_s} W_t \right)^2= \bluepassage{c^2_H/(2H)} (t-s)^{2H}$ a fact that we will use in Section 3 repeatedly.


We will also make use of Clark-Ocone formula (see \cite{Nua06intro}). For $s \geq 0$, let $X$ be a \redpassage{$\FF_s$-measurable} \bluepassage{$\sigma((B_z)_{z \in \RR})$-measurable } random variable $X$ with Malliavin derivative $D_{\cdot} X$. Then the Clark-Ocone formula states that (under suitable integrability conditions):

\[ X = \EE_{s} X + \int_s^{\infty} \EE_{z} \left[D_z X\right] dB_z \]

If additionally we assume that $X$ is $\FF_u$-measurable for $u > s$, then:

\[ X = \EE_s X + \int_s^u \EE_z \left[ D_z X \right] dB_z \]

For $f \in \ccc{1}$ and $\wh_t = \int_{-\infty}^t k(t,z) dB_z$ the Malliavin derivative of $X=f(\wh_t), s < t$ is given by

\begin{equation}\label{eq:malliavinfdx}
D_s f(\wh_t) = \nabla f(\wh_t) \cdot k(t,s) 
\end{equation}


\bluepassage{
Throughout the paper, we will work with a $d$-dimensional fBm $W$, which simply means $W = (W^1,\ldots, W^d)$ where the $W^i$ are independent scalar fBms. For ease of notation, we will always keep the summation over coordinates implicit and write all the computations as if we were in the scalar case. Since none of the arguments depend on the dimension, this should not lead to any confusion.}
\subsection{Nonlinear integration}

Since averaged fields will be understood as nonlinear Young integrals, we will state their existence in $p$-variation. The proofs use deterministic sewing lemma for general controls. In particular, for any given path $X$ we will always understand $\omega_X(s,t) = \qvarst{X}^q$ and use it as a control function.

We will define the nonlinear Young integral along maps $A =A(t,x) \in \vvv{q}_t\ccc{1}_{x}, q <2$. 


%

\color{black}
In this subsection we will always fix $q < 2$. A sketch of proof of well-posedness is given in Proposition \ref{lem:integrest} below. 

\begin{definition}\label{def:nyoungint}
Let $A \in \vvv{q}_t\ccc{1}_x$. Then for any $\theta \in \vvv{q}$:

\[ \int_0^t A(dr, \theta_r) = \lim_{\absv{\PC} \to 0} \sum_{[u,v] \in \PC} A_{uv}(\theta_u) \]

where $\PC$ is taken among partitions of $[0,t]$.
\end{definition}

 
\begin{proposition}\label{lem:integrest}
Under the assumptions on $A,\theta$ in \ref{def:nyoungint} the nonlinear Young integral is well-defined and comes with the following bound, where $C$ is a generic constant:

\begin{equation}\label{eq:indefnyoung}
\qvarst{ \int_0^{\cdot} A(dr,\theta_r) } \leq 2C \left(  \qvarstCx{A} \qvarst{\theta} +  \qvarstCx{ A} \absv{\theta_0 - A(0,0)} \right)
\end{equation}

Analogously, for linear integral, where $W, B \in \vvv{q}$ are real-valued paths:
\begin{equation}\label{eq:indeflyoung}
\qvarst{ \int_0^{\cdot} B_r dW_r  } \leq 2 C\left( \qvarst{ W } \qvarst{ B} + \absv{ B_s } \qvarst{W} \right)
\end{equation}

\end{proposition}

\begin{proof}
Define $Z_{st} = A_{st}(\theta_s)$. Then it holds that $\absv{ \delta_u Z_{st} } \leq \qvarstCx{A}  \qvarst{\theta}$. Since $q < 2$ we have that this product is a control, so that we can apply sewing lemma for general controls, see Lemma 2.2 in \cite{DGHT19}. The limiting path is our nonlinear Young integral and the estimate \eqref{eq:indefnyoung} follows from 
\begin{equation}\label{eq:controlbdyoung}
\absv{ \int_s^t A(dr, \theta_r) - A_{st}(\theta_s) } \leq \qvarstCx{A}  \qvarst{\theta}.
\end{equation}

This implies, by the spatial Lipschitz property of $A$ that:

\[ \qvart{ \int_0^{\cdot} A(dr,\theta_r)} \leq \qvart{ A } \qvart{ \theta} + \qvart{ A } \norm{ \theta }_{\infty;[0,t]} \]

Estimate \eqref{eq:indeflyoung} follows by identical reasoning - $A(t,x) = x W_t$, so for any subinterval $[u,v] \subset [s,t]$ we have:

\[ \absv{ \int_u^v B_r dW_r } \leq \norm{ W }_{q;[u,v]} \norm{ B}_{q;[u,v]} + \norm{ B }_{\infty;[s,t]} \norm{ W }_{q;[u,v]} \]

and then desired bound on variation norm follows. Once we obtained both variation estimates, we can observe $\norm{ B }_{\infty;[s,t]} \leq \absv{ B_s } + \qvarst{ B}$ and get the bound in terms of initial condition.
\end{proof}

\bluepassage{We now state a Gronwall-type estimate that we will need in the sequel. Note that this is not exactly a standard one, due to the fact that the left-point of the interval is kept fixed at $0$.}


\bluepassage{

\begin{lemma}\label{lem:stabest}
Fix $K > 0$ and $A \in \ccc{q-var}_t\ccc{1}_x$ for some $q<2$, with $A(t,0) = 0$. Let $Y$ be a path satisfying:

\[ \forall t \geq 0, \qvart{Y} \leq \qvart{ \int_0^{\cdot} A(dr, Y_r) } + K \]

Then for arbitrarily small $\eps > 0$ and some constants $C_1, C_2 > 1$:

\[ \forall t \geq 0, \norm{ Y }_{q-var;[0,t]} < C_1 \exp\left( C_2 \qvartCx{ A}^{q(1+\eps)} \right) \left( \absv{Y_0}+ K \right)  \]

\end{lemma}

\begin{proof}
See Appendix \ref{subsec:variationyoung}
\end{proof}

}

\redpassage{
\begin{lemma}\label{lem:stabest}
Let $A \in \vvv{q}_t, h \in \vvv{q}$\bluepassage{, where $q<2$}. Let $Y$ be a path for which the following inequality is satisfied:

\begin{equation}\label{eq:varineq}
\qvarst{ \delta Y_{\cdot} } \ls  \absv{ Y_s } + \qvarst{ \int_0^{\cdot}  dA_r Y_r} + \qvarst{ h_{\cdot} }
\end{equation}

then we have the following global estimate:

\[ \qvart{Y} \leq \exp( C_{q,R} \qvart{A}^q ) (\absv{Y_0} + \qvart{h}) \]

where $C_{q,R}$ is a constant depending on the implicit constant in \eqref{eq:varineq} and $q$.
\end{lemma}

\begin{proof}
See Appendix
\end{proof}

}

Finally, we recall here the linearisation of difference of two nonlinear integrals, integrated with respect to the same map $A$.

\begin{lemma}\label{lem:linearisation}
Let $A$ be a map $A \in \ccc{\gamma}_t\ccc{1+\nu}_x$ such that $\gamma(1+\nu) > 1$. Then for paths $\theta,\bar{\theta} \in \ccc{\gamma}$:

$$\int_0^t A(dr,\theta_r) - \int_0^t A(dr,\bar{\theta}_r) = \int_0^t \left( \theta_r - \bar{\theta}_r \right) dV_r$$

where the integral on right hand-side is understood as an linear Young integral, $V \in \ccc{\gamma}$ and:

$$V_{\cdot} = \int_0^1 \int_0^{\cdot} \nabla A(ds, \theta_s + x(\bar{\theta}_s - \theta_s) ) dx$$
\end{lemma}

\begin{proof}
See Lemma 6 in \cite{GG20}.
\end{proof}

\subsection{Averaged fields}

As indicated in the introduction, one of the crucial elements of the proof is the consideration of the averaged field $T^w f$,   \bluepassage{which we define here for $w$ an arbitrary $\RR^d$-valued path}. This definition is intrinsic and does not depend on any mollification procedure, nor it is reliant on the probabilistic properties of a path $w$, \bluepassage{which is why we abstain from using $T^{\wh}$ or $T^B$ and introduce a separate notation}. In the simple case of a continuous $f \in \ccc{0}_{t,x}$, the averaged field coincides with the following integral:

\[ (t,x) \mapsto T^w f_t(x) := \int_0^t f_r(w_r + x) dr \]

Now let us consider the case where $f \in L^1E$ where $E$ is a Banach space continuously embedding in $\mathcal{S}'(\RR^d)$, with norm invariant under translation. Then the averaged field can be defined using translations $\tau^w: f \mapsto \tau^w f = f(\cdot + w)$ as a Bochner integral. 

\begin{definition}\label{def:averagedfield}
Let $w: [0,T] \mapsto \RR^d$ be a measurable function and $E$ as above. Then for fixed $f$ we define the averaged field $T^w$ as the linear map from $L^1E$ to $\ccc{0}E$ 

\[ T^w_t f := \int_0^t \tau^{w_s} f_s ds\;\; \forall t \in [0,T] \]
\end{definition}

It can also be useful to describe averaged fields by duality - in the sense of distributions, the dual of translation is $(\tau^w) = \tau^{-w}$ and we have for any $\varphi \in \mathcal{S}(\RR^d) \hookrightarrow E^*$:

\[ \mean{ T^w f, \varphi } = \int_0^t \mean{ f_s, \varphi(\cdot - w_s) } ds \]

The advantage of this relation is that it is now easy to show that spatial differentiation and averaging commute.

We also remark that averaged fields are continuous with respect to $f$, in the sense of the following proposition:
\begin{proposition} \label{prop:ConvAF}
Let $f \in L_t^qE$ and let $(f^n)_{n \geq 0}$ be a sequence of functions such that $\norm{ f^n - f }_{L^1_t E} \rightarrow_n 0$. Then $T^w f^n \to T^w f$ in $\ccc{0}_tE$.

In addition, if each of the $T^w f^n$ are continuous functions, and are locally equicontinuous (i.e. share a modulus of continuity on each compact), then $T^W f$ is also continuous, with the same local modulus of continuity.
\end{proposition}

\begin{proof}
By linearity of the Bochner integral $T^w f^n - T^w f = \int_0^t \tau^{w_s} \left(f^n_s - f_s \right)ds$, i.e. $\tau^{w_s}$ is now applied to the tempered distribution $f^n_s - f_s$. It follows:

\[ \norm{ T^w (f^n - f) }_{\ccc{0}E} \leq \norm{ f^n - f }_{L^1E} \to 0 \]

For the second part, note that by the Arzela-Ascoli theorem, it holds that, taking a subsequence if necessary, $T^w f^n$ converges locally uniformly to a continuous function $g$. But then, by the first part, $g$ and $T^w f$ are equal (as elements of $\mathcal{S}'(\RR^d)$) and the result follows.

\end{proof}

For more details and extended discussion on averaging operators we refer the reader to \cite{GG20}, Section 3.1.

\section{Regularity of the averaged field with perturbation}\label{sec:avgpert}

In this section, we consider the regularity properties of the averaged field for a fractional Brownian motion perturbed by an adapted process with sufficient (variation) regularity. The main result is the following.

\begin{theorem}\label{thm:continuity}
Let $W^H$ be a fractional Brownian motion with a Hurst index $H$ and consider the extended filtration $\FF$ from \eqref{eq:mandelbrot}. Then let $f \in \fsp$ and $\phi$ an $\FF_t$-adapted process whose increments can be bounded by $\absv{\phi_{st}} \leq \omega_{\phi}(s,t)^{\beta}$ for some $\beta \in (H,1]$ where $\omega_{\phi}$ a control function, which is adapted (in the sense that $\omega_{\phi}(s,t)$ is $\FF_t$-measurable for $s\leq t$). Pick $\xi \in (0,1]$, assume that 
\begin{equation}\label{cond:mgale}
\epsilon := 1/2+H(\alpha-\xi ) - \inv{p} >0,
\end{equation}
let
\begin{equation}\label{cond:integr}
\gamma := 1 - \inv{p} + H(\alpha-\xi-\eta)
\end{equation}

and fix $\eta \in [0,1]$ s.t.

\begin{equation}\label{cond:sewing}
\gamma + \eta \beta > 1.
\end{equation}

%
%

Then for each $\delta \in  (1/2, 1/2+\epsilon)$, $R>0$, there exist positive constants $c, \Lambda$ and a random variable $C_{\delta,R}$ such that it holds that for all pairs $s<t$ and finite $R>0$ 
 \begin{align*} \sup_{\absv{x}  < R} \frac{ \absv{\int_s^t T^{\wh+\phi} f(dr,x) }}{\nofn } + 
  &\sup_{x \neq y, \absv{x} + \absv{y} < R} \frac{ \absv{\int_s^t T^{\wh+\phi} f(dr,x) - \int_s^t T^{\wh+\phi} f(dr, y) }}{\nofn \absv{x-y}^{\xi}}\\
  & \leq c (t-s)^{\gamma} \omega_{\phi}(s,t)^{\eta \beta} + C_{\delta,R}(t-s)^{\delta} 
  \end{align*}

where $C_{\delta,R}$ additionally satisfies
  \[ \EE[\exp(\Lambda C_{\delta,R}^2)] < \infty.\]

\end{theorem}


\begin{remark}
Let us comment on the assumptions made on the parameters in the Theorem. First note that \eqref{cond:sewing} implies that $\gamma >0$ since $\eta, \beta \leq 1$.

Consider the case when $\eta=1$. In this case, the theorem says that, a.s., $T^{W^H+\phi}f$ takes values in $C^{q-var}_t C^{\xi}_x$, for some $q<2$, as long as
\[  \xi < \alpha - \frac{1}{pH} +   \left(\frac{1}{2H} \right) \wedge \left(\frac{\beta}{H} - 1\right)  \]
If the regularity $\beta$ of $\phi$ is such that $\beta > H + 1/2$, the r.h.s. is equal to $\alpha - \frac{1}{pH}  + \frac{1}{2H}$ , which is the same as the spatial regularity of $T^{W^H} f$ (as obtained in \cite{CG16}). 

The general statement above where we can choose $\eta \in [0,1]$ (at a cost in the attained regularity $\xi$) will be needed to obtain Gaussian integrability (as required e.g. to apply Girsanov's theorem), in cases where $\omega_\phi$ does not have Gaussian tails (this is the case for reflection measure of fBm, see Proposition \ref{thm:expintegrbox}).
\end{remark}

Our basic building block is going to be the following germ \eqref{eq:bblock}, defined for smooth and bounded functions $f$, where $P_{\sigma^2(s,r)}$ is heat semigroup with variance $\sigma^2(s,r)$ and $W^{2,H}_{sr}$ is the $\FF_s$-measurable part of $\wh_{sr}$, see subsection \ref{subsect:fbm}. Filtration $\FF_s$ is the filtration of the two-sided Brownian motion in \eqref{eq:mandelbrot}.

\begin{equation}\label{eq:bblock}
A_{st}[f](x) = \int_s^t \EE_s f_r(W^H_r + \phi_s + x) dr = \int_s^t (P_{\sigma^2(s,r)} f_r)(W^{2,H}_{sr} + \phi_s + x) dr
\end{equation}

In other words, we integrate out the part of $\wh_r$ that is independent of $\FF_s$, see the proof of Lemma 6.4 in \cite{Le20} for details. 

The proof of Theorem \ref{thm:continuity} will then be obtained by a variant of the stochastic sewing lemma, applied to this germ. 
Let us give a sketch of the proof. Since we want to apply a sewing procedure, we consider for $s\leq u \leq t$ the quantity $\delta_u A_{st}  (= A_{st} - A_{su}-A_{ut}$), which we compute as
\begin{align*}
\delta_u A_{st} &= \int_u^t \EE_s f(\wh_r+\phi_s+x) dr - \int_u^t \EE_u f(\wh_r+\phi_u+x) dr \\
&=  \int_u^t \left( \EE_u f(\wh_r+ \phi_s+x) - \EE_u f(\wh_r + \phi_u +x) \right)dr + M_{st}^u, 
\end{align*}
where $M_{st}^u$ satisfies $\EE_s M_{st}^u = 0$.

Using the fact that $\phi_s,\phi_u$ are $\FF_u$-measurable, we can, using heat kernel estimates, bound the first term by $(t-s)^{\gamma} w_{\phi}(s,t)^{\eta \beta}$ with $\gamma + \eta \beta > 1$, $w_{\phi}$ being the control function of $\phi$. In addition, by an application of Clark-Ocone's formula and BDG inequality, $M_{st}^u$ turns out to have small enough moments, that is bounded by $\ls (t-s)^{1/2+\eps}$.

These two bounds make an argument similar to the stochastic sewing lemma in the sense of \cite{Le20} possible (more precisely, we use a variant with a random control as in \cite{ABLM21}, see Theorem 4.7 therein), \bluepassage{in Lemma \ref{lem:aacsewn} we construct a path $t \mapsto \mathcal{A}_t$, which enjoys the following bound on its increments:} \redpassage{after which we recover}:

\[ \absv{ \AAC_{st} } \leq (t-s)^{\gamma} w(s,t)^{\eta \beta} + B_{st} \]

where $B_{st}$ is defined as $L^m$-integrable (for every $m \geq 1$) map, corresponding to the sum of martingale terms. By a technical chaining argument we then convert $B$ to a.s. H\"older-type bounds.

\vspace{5mm}

Let us now proceed with the actual proof. In order to show spatial regularity of the averaged field, we will also consider the difference of \eqref{eq:bblock}, evaluated at two different points $x,y$.

In order to show spatial regularity of averaged field, we will also consider the difference of \eqref{eq:bblock}, evaluated at two different points $x,y$. \bluepassage{The following lemma treats the case of $\xi \in (0,1)$. The case $\xi=1$ follows from Corollary \ref{cor:differentiable_fields} below}.

\begin{lemma}\label{lem:asutdecomp}
\bluepassage{ Let $f \in L_t^{\infty}\ccc{\infty,b}, \xi \in (0,1)$. For any $x, y \in \RR^d, x \neq y, (s,u,t) \in \Delta^3$ the difference of germs defined in \eqref{eq:bblock} can be decomposed as:

\[ \delta_u \left( A_{st}[f](x) - A_{st}[f](y) \right) = I^{x,y}_{sut} + M^{x,y}_{sut} \]

with $M_{sut} = \int_s^t \bar{J}^{s,u,t}_z d\nois_z$ and \[ \absv{ \bar{J}^{s,u,t}_z } \leq C \absv{x-y}^{\xi} \absv{t-s}^{1/2+H(\alpha-\xi)- \inv{p}} \nofn \]

and \[ \absv{ I_{sut} } \leq \absv{x-y}^{\xi} \nofn (t-s)^{1+H(\alpha+\eta-\xi) - \inv{p}} w(s,t)^{\eta \beta} \]

where $(\alpha,\xi)$ are linked by the inequality $\alpha > \xi - \inv{2H} + \inv{pH}$
}

\redpassage{
Let $f \in L_t^{\infty}\ccc{\infty,b}_x$. Then for any two $x,y \in \RR^d, x \neq y$, the germ defined in \eqref{eq:bblock} can be decomposed as:

$$\delta_u A_{st}[f](x) = I_{sut} + M_{sut}$$

where $$\forall (s,t) \in \Delta^2, \;\ I_{sut} \ls \nofn (t-s)^{\gamma} w(s,t)^{\eta \beta}$$
and  $\quad M_{sut} = \int_s^t \bar{J}^{s,u,t}_z dW_z$ where $\bar{J}^{s,u,t}$ can be estimated in the following way:

$$\absv{ \bar{J}^{s,u,t} } \leq C \absv{x-y}^{\xi} \absv{t-s}^{1/2+H(\alpha-\xi)- \inv{p}} \nofn$$

Moreover, for a difference of germs we have, for $k \in \{0,1\}$

$$\delta_u \left( A_{st}[\nabla^k f](x) - A_{st}[\nabla^k f](y) \right) = I_{sut}^{x,y} + M^{x,y}_{sut}$$

with:

$$\absv{ I^{x,y}_{sut} } \leq \nofn \absv{x-y}^{1 \land \xi-k} (t-s)^{\gamma} w_{\phi}(s,t)^{\eta \beta} \quad   \norm{ M^{x,y}_{sut} }_m \leq \sqrt{m} \absv{x-y}^{1 \land \xi-k} \nofn (t-s)^{1/2+}$$

and $\xi$ taken such that $\alpha > \xi - \inv{2H} + \inv{pH}$.

}
\end{lemma}

\begin{proof}
Computations below follow by Clark-Ocone formula and stochastic Fubini theorem. Recall the decomposition $\wh_r = \left( \wh_r - \EE_{\FF_s} \wh_r \right) + \EE_{\FF_s} \wh_r$ and for $f$ as assumed:

\[ \EE_s f(\wh_r+ \phi_s + x) = (P_{\sigma^2(s,r)} f)(\EE_s \wh_r + \phi_s + x) \]

First let us specify how we apply \eqref{eq:malliavinfdx} to our use case. Let \bluepassage{$X$ be the two-index map

\[ X_{vt} = \int_v^t \EE_v f_r(W^H_r + \phi_v + x) dr \]
}and pick $z < v$. Then $\EE_z \EE_v = \EE_z$ by tower property and 

\begin{equation}\label{eq:clarkocone}
 \EE_z D_z X_{vt} = \int_v^t (P_{\sigma^2(z,r)} \nabla f_r)(\EE_z \wh_r + \phi_s + x) (r-z)^{H-1/2} dr 
\end{equation}

We combine \eqref{eq:clarkocone} with the stochastic Fubini theorem: $\int_X \int_0^t \psi(s,x) d\nois_s dx = \int_0^t \int_X \psi(s,x) dx d\nois_s$ for some $\psi$ jointly measurable and adapted (see \cite{Ver11Fubini} for a self-contained proof).

\begin{align*}
\absv{x-y}^{\xi} \delta_u A^{x,y}_{st}[f]  =&  \int_u^t \EE_s f_r(\wh_r + \phi_s + x) dr - \int_u^t \EE_s f_r(\wh_r + \phi_s + y) dr \\
& - \int_u^t \EE_u f_r(\wh_r + \phi_u + x) dr + \int_u^t \EE_u f_r(\wh_r + \phi_u + y) dr \\
= & \int_u^t \EE_u g_r(x,y,\phi_s,\phi_u) dr + \bluepassage{\int_s^u \EE_z D_z \left( \int_u^t f_r(\wh_r+\phi_s+x) - f_r(\wh_r+\phi_s+y) dr\right) d\nois_z } \\ & \redpassage{\int_u^t \int_s^u \EE_z D_z \left( f_r(\wh_r + \phi_s + x) - f_r(\wh_r + \phi_s + y) \right) dr d\nois_z } \\ & =: I^{x,y}_{ut} + M^{x,y}_{ut} 
\end{align*}

with 
\begin{equation}\label{eq:grhelp}
g^{\wh}_r(x,y,a,b) = f_r(\wh_r + x + a) - f_r(\wh_r + y + a) - f_r(\wh_r + b +x) + f_r(\wh_r + b + y)
\end{equation}

Let us begin from the first term. Applying Proposition \ref{prop:hoelderdiff} to $g_r$ we obtain that 

$$\forall\ (x,y,a,b)\ \absv{ g_r } \leq \absv{ x- y}^{\xi} \absv{ a-b}^{\eta} \norm{ f_r }_{\ccc{\xi+\eta}} $$

We then apply expected value and applying estimate above directly to the function $P_{\sigma^2(u,r)} g^{\EE_u \wh_r}_r$:

\[ \EE_u g^{\wh}_r = (P_{\sigma^2(u,r)} g^{\EE_u \wh_r}_r)(x,y,\phi_s, \phi_u) \leq \absv{ x- y}^{\xi} \absv{ \phi_{su}}^{\eta}  \norm{ P_{\sigma^2(u,r)} f_r }_{\ccc{\xi+\eta}} \]

Now it follows by local non-determinism and heat kernel estimates (Prop. \ref{prop:heatkernel}) 
\begin{align*}
I^{x,y}_{ut}  \leq  \absv{x-y}^{\xi} \absv{ \phi_{su} }^{\eta} \int_u^t \norm{ P_{\sigma^2(u,r)} f_r }_{\ccc{\xi+\eta}} dr \leq \absv{x-y}^{\xi} \absv{ \phi_{su} }^{\eta}  (t-u)^{1+H(\alpha-\xi-\eta) -\inv{p}} \nofn
\end{align*}

It shows the announced bound for the first term, taking $\gamma = 1 + H(\alpha-\xi-\eta)- \inv{p}$ and $\absv{ \phi_{su} }^{\eta} \leq w_{\phi}(s,u)^{\eta \beta}$. For the second term we can immediately apply BDG inequality. First set $\bar{f}_r = f_r(\cdot+x) - f_r(\cdot+y)$ and by \eqref{eq:clarkocone} $\bar{J}_z^{s,u,t} = \int_u^t P_{\sigma^2(z,r)} \nabla \bar{f}(\EE_z \wh_r + \phi_s) \bluepassage{(r-z)^{H-1/2}} dr$ (we suppress dependence on $x,y$ in notation for brevity).  Again using Proposition \ref{prop:heatkernel} it follows that, \bluepassage{for $\xi \in (0,1)$ and using $z < u$}:

\bluepassage{
\begin{equation}\label{eq:barjsut}
\begin{aligned}
\absv{ \bar{J}^{s,u,t}_z } \leq & \int_u^t \norm{ P_{\sigma^2(z,r)} \nabla \bar{f}_r}_{L^{\infty}} (r-z)^{H-1/2} dr \\
\leq & \absv{x-y}^{\xi} \int_u^t \norm{ P_{\sigma^2(z,r)} \nabla f_r}_{\ccc{\xi}} (r-z)^{H-1/2}  dr \\
\leq & \absv{x-y}^{\xi} \int_u^t (z-r)^{H(\alpha-\xi)-1/2} \norm{ f_r }_{\ccc{\alpha}_x} dr \\
\leq & \absv{x-y}^{\xi} \absv{s-t}^{H(\alpha-\xi)+1/2 - \inv{p}} \norm{ f }_{L^p_t\ccc{\alpha}_x} 
\end{aligned}
\end{equation}
}

where we used $H(\alpha-\xi)>1/2, z > s$ and the H\"older inequality in the last line.

\redpassage{
\begin{align*}
\absv{ \bar{J}_z^{s,u,t}  } \leq & \int_u^t \norm{ P_{\sigma^2(u,r)} \nabla \bar{f} }_{L^{\infty}} (r-u)^{H-1/2} dr \\
\leq & \int_u^t \absv{ x- y}^{\xi} \norm{ P_{\sigma^2(u,r)} \nabla f }_{\ccc{\xi}} (r-u)^{H-1/2} dr \\
\leq &\absv{ x- y}^{\xi}  \int_u^t (r-u)^{H(\alpha-\xi) - 1/2} \norm{ f_r }_{\ccc{\alpha}_x} dr \\
\leq & \absv{ x- y}^{\xi} (t-u)^{1/2+H(\alpha-\xi)-\inv{p}} \nofn
\end{align*}
}

In the end we obtain \bluepassage{with $C_p \simeq p^{p/2}$ for $p$ large enough}

$$\EE \left[ \absv{ M^{x}_{ut} - M^y_{ut} }_p \right] \ls C_p \EE \left[ \left( \int_s^t \absv{ \bar{J}^{s,u,t} }^2 dz \right)^{p/2} \right] \ls C_p \absv{x-y}^{p \xi} (t-s)^{p\left(1+H(\alpha-\xi) - \inv{p} \right) }  $$

which in turn shows the claim about the second term. 

\redpassage{
To show desired estimates for $\nabla f$, we note that for smooth $f$ there is $\nabla A_{st}[f] = A_{st}[\nabla f]$, moreover if $f \in \ccc{\alpha}$, then $\nabla f \in \ccc{\alpha-1}$ in a weak sense. Then for any $\alpha$ we substitute $\alpha - 1$ and then, if we want to have $\xi > 1$, we take $\xi - 1$ instead of $\xi$.
}
\end{proof}

\bluepassage{
\begin{corollary}\label{cor:differentiable_fields}
Let $\xi > 1, k = \lfloor \xi \rfloor$. Then:

\[ \absv{  \nabla^k A_{st}[f](x) - \nabla^k A_{st}[f](y)  } \leq I^k + M^k \]

where \[ I^k \leq \absv{x-y}^{\xi-k} \nofn \absv{t-s}^{\inv{2} + H(\alpha - \xi) - \inv{p}} \] and $M^k = \int_s^t J_z d\nois_z$ with \[ \absv{ J_z^{sut}  } \leq \absv{x-y}^{\xi-k} \nofn \absv{t-s}^{\inv{2} + H(\alpha+\eta-\xi) - \inv{p}} w_{\phi}(s,t)^{\eta \beta}  \]
\end{corollary}

\begin{proof}
For smooth $f$ there holds $\nabla A_{st}[f] = A_{st}[\nabla f]$. Moreover, we have that $f \in \ccc{\alpha}$ implies $\nabla^k f \in \ccc{\alpha-k}$. Therefore in the proof above we substitute $f$ with $\nabla^k f$ and the desired estimates follow directly.
\end{proof}

}

The following lemma is then an application of stochastic sewing lemma techniques to the obtained estimates. 
%
%
%

\begin{lemma}\label{lem:aacsewn}
Let $f \in L_t^{\infty}\ccc{\infty,b}_x$. \bluepassage{For any $x \in \RR^d$, } there exists a unique process $t \mapsto \AAC_t$, constructed as:

\[ \AAC_t(x) = \lim_{\absv{\PC} \to 0} \sum_{[u,v] \in \PC} A[f]_{uv}(x)  \]

where $\PC$ is an arbitrary partition, $A$ is defined as \eqref{eq:bblock}, the limit \bluepassage{holds in $L^m(\Omega)$ for arbitrary $m\geq 2$ and } does not depend on the choice of partition, and we have the following estimates (for $\xi, \alpha, \gamma, \eta, \delta$ as in Theorem \ref{thm:continuity}):

\bluepassage{
\begin{equation}\label{eq:sewingbd}
\absv{ \AAC_{st}(x) - \AAC_{st}(y) - A_{st}[ f](x) - A_{st}[ f](y) } \leq \absv{x-y}^{\xi} \nofn (t-s)^{\gamma} w_{\phi}(s,t)^{\eta \beta} + U^{x,y}_{st} 
\end{equation}
}

\redpassage{
\begin{equation}\label{eq:sewingbd}
\absv{ \AAC_{st}(x) - \AAC_{st}(y) - A_{st}[\nabla f](x) - A_{st}[\nabla f](y) } \leq \absv{x-y}^{\xi-1} \nofn (t-s)^{\gamma} w_{\phi}(s,t)^{\eta \beta} + U^{x,y}_{st} 
\end{equation}
}

where $U^{x,y}_{st}$ satifies 
\begin{equation}\label{eq:triangle}
\forall\ x\neq y,\ \forall\ s<v<t,\;\; \absv{U^{x,y}_{st} } \leq \absv{ U^{x,y}_{sv} } + \absv{ U^{x,y}_{vt} }
\end{equation}

\begin{equation}\label{eq:moments}
 \norm{ U^{x,y}_{st} }_m \leq C_m \nofn \absv{x-y}^{\xi} \absv{t-s}^{\delta'} 
\end{equation}

\bluepassage{with $C_m \sim m^{1/2}$ for $m \to \infty$ and $\delta' > \delta$}

\end{lemma}

\begin{proof}
Existence of the limit follows from Lemmas \ref{lem:existence} (dyadic partition) and then \ref{lem:existall} for arbitrary partition. The limit exists as a random field $(t,x) \mapsto \AAC_t(x)$ in $L^m, m \geq 2$, which enjoys the following bound:

\[ \absv{ \AAC^{x,y}_{st} - A_{st} } \leq \bluepassage{\absv{x-y}^{\xi}} w_{\phi}(s,t)^{\eta \beta} (t-s)^{\gamma} + B^{x,y}_{st} \]

$B$ obtained in this way does not necessarily satisfy \eqref{eq:triangle}, however by Lemma \ref{lem:aacpathwise} there exists a map $U^{x,y}_{st}$ such that $\absv{ U^{x,y}_{st} } \leq \absv{ B^{x,y}_{st} }$ and $B$ in the bound above can be replaced by $U$. As $B$ satisfies \eqref{eq:moments} it follows that it is the case for $U$ as well.

\end{proof}

With the help of a suitable chaining argument we will apply the construction of $\AAC$ with the above Lemma \ref{lem:aacpathwise} to prove Theorem \ref{thm:continuity}. At this point it is not clear that the limit path $\AAC$ coincides with $T^{\wh+\phi} f_t(x) = \int_0^t f_r(\wh_r + \phi_r + x) dr$. The following lemma clears this confusion.

\begin{lemma}\label{lem:ident}
Let $\AAC$ be the limiting path obtained in Lemma \ref{lem:aacsewn} for $f \in L_t^{\infty}C_x^{\infty,b}$. Then the limit is identical to the classical integral.

\[ \forall\;(t,x) \in \RR_+ \times \RR^d\;\; \AAC[f](x)_t = \int_0^t f_r(\wh_r+\phi_r+x) dr \]

\end{lemma}

\begin{proof}

We will use a uniqueness result for stochastic sewing. Fix $x$ and let 
 $\bar{\AAC}_{st} = \int_s^t f_r(\wh_r + \phi_r + \cdot)dr$, which we want to show to be equal to $\AAC_{st}$. It holds that
\begin{equation}\label{eq:twodifflim}
\bar{\AAC}_{st} - \AAC_{st}  =  (\bar{\AAC}_{st} - A_{st}) + ( A_{st} - \AAC_{st})
\end{equation}

Our goals is to show that the above can be further bounded by $I_{st} + M_{st}$, where $\absv{ I_{st} } \leq w(s,t)^{1+\eps}$ for some random control $w$ and $M_{st}$ is a sequence of martingale differences. Then we can apply Lemma \ref{lem:sewuniq}. The second pair in \eqref{eq:twodifflim} obviously satisfies the required bound by existence lemma \ref{lem:existence} - we have that ${\AAC}_{st} - A_{st} = {I}_{st} + {N}_{st}$. Observe that it suffices to take germ $A_{st}[f](x) = \int_s^t \EE_s f_r(\wh_r+\phi_r + x) dr$ rather than $A_{st}[f](x) - A_{st}[f](y)$ since by regularity of $f$ we will immediately have spatial regularity of the averaged field, i.e. $\absv{ \AAC[f](x) - \AAC[f](y) } \leq \norm{ \nabla f }_{\infty} \absv{x-y}$. for all $x,y$.

We intend to compare $A_{st}$ with $\AAC[f](x)$ as in the statement of the lemma. We then use the difference between $f(\wh_r + \phi_s + x)$ and $f(\wh_r+\phi_r+x)$ to make the term bounded with $w_{\phi}(s,t)$ appear.

\begin{align*}
\left( \bar{\AAC}_{st} - A_{st}\right) = &  I^{1}_{st} + I^{2}_{st} \\ := & \int_s^t \left( f_r(\wh_r + \phi_r + x) - f_r(\wh_r + \phi_s + x) \right) dr  +  \int_s^t\left( f_r(\wh_r + \phi_s + x) -  \EE_s  f_r(\wh_r + \phi_s + x) \right) dr.
\end{align*}

Then we use $I^1_{st}$ can be bounded using $f \in \ccc{1}$:

\[ \absv{ I^1_{st} } \leq \int_s^t \norm{ \nabla f_r }_{L^{\infty}} \absv{ \phi_{sr} } dr \leq \absv{ \phi_{st}} \int_s^t \norm{ \nabla f_r }_{L^{\infty}} dr \leq (t-s) w_{\phi}(s,t) \norm{ f }_{\ccc{1}} \]

Clearly, this is of the form $w(s,t)^{1+\eps}$ for a certain control $w$.

We have therefore the first part of desired statement. For $I^{2}_{st}$, it is clear that it satisfies $\EE_s I^{2}_{st} = 0$, and since $f$ is bounded, it also holds that $| I^{2}_{st}  | \leq C (t-s)$. We can therefore apply Lemma \ref{lem:sewuniq} to conclude that $ \bar{\AAC}_{st} =  {\AAC}_{st}$ a.s. for all $s<t$.
\end{proof}

We now have all the necessary ingredients to finish the proof of the regularity properties of $T^{\wh+\phi} f$.

\begin{proof}[Proof of Theorem \ref{thm:continuity}]
Take $f \in L_t^{\infty}\ccc{\infty,b}_x$ - this assumption will be removed at the end of the proof. The well-posedness of $\AAC$ follows from Lemma \ref{lem:aacsewn}, and using Lemma \ref{lem:aacpathwise} we conclude that there exists a random map $U: \Delta^2 \times \RR^d \times \Omega \mapsto \RR^d$ such that:

$$\absv{ \AAC_{st}(x) - \AAC_{st}(y) } \leq w(s,t)^{1+\eps} + U^{x,y}_{st} $$

and $(s,t) \mapsto U^{x,y}_{st}$ satisfies property \eqref{eq:triangle}. By construction of $\AAC$ and \eqref{eq:yupb} we have that $U^{x,y}_{st} \leq \sum_{k \geq 0} \absv{ J^k(x)_{st} - J^k(y)_{st} }$, where $J^k(x)_{st} = \int_s^t h^k(x)_z dW_z$, where $h$ is defined in \eqref{eq:jstk}. Therefore $\norm{ U^{x,y}_{st} }_m \leq \sqrt{m} \nofn (t-s)^{\delta} (x-y)^{\xi}, \delta > 1/2$ and we can use Lemma \ref{lem:chain} to conclude that there exists a random constant $C_{\delta}$ such that $\absv{ U^{x,y}_{st} } \leq C_{\delta} \absv{t-s}^{\delta} \absv{ x-y }^{\xi}$ and $\norm{C_{\delta}}_q \leq  C q^{1/2} \nofn$ for some determinisitic constant $C$ independent of $\delta$. Then it follows by a standard Taylor expansion in conjuction with Stirling approximation that there exist $\bar{\lambda} > 0$ such that $\EE \left[ \exp\left( \bar{\lambda} \frac{ C_{\delta}^2 }{ \nofn^2 } \right) \right] < \infty$.

Therefore we have the following bound for smooth $f$, taking any finite $R>0$:

\[\sup_{\substack{x \neq y \\ \absv{x} + \absv{y} < R}} \frac{\absv{ \AAC_{st}[f](x) - \AAC_{st}[f](y) }}{\absv{x-y}^{\xi}} \leq c (t-s)^{\gamma} \omega_{\phi}(s,t)^{\eta \beta} + C_{\delta,R} (t-s)^{\delta} \]

By Lemma \ref{lem:ident} we show that for $f$ smooth the limit of sewing procedure is $\AAC[f](x) = \int_0^t f_r(\wh_r+\phi_r+ x) dr$, so that it is exactly the averaged field. Take now arbitrary $f \in L_t^q \ccc{\alpha}_x$, and let $f^n$ be smooth and converging to $f$ in $f \in L_t^{q} \ccc{\alpha-\epsilon}_x$, for any $\epsilon>0$. Each of $T^{W+\phi} f_n$ satisfies the above bound with the same $c$ and some r.v. $C_{\delta,R}^n$. Then it follows from Lemma \ref{prop:ConvAF} that this also holds for $f$, with $C_{\delta,R}= \liminf_n C_{\delta,R}^n$, and Gaussian integrability of $C_{\delta,R}$ is immediate from Fatou's lemma.

\end{proof}

%
%

\section{Fixed point equations}\label{sec:FixedPointEqs}

In this section we show how the regularity obtained in the previous Lemma allows to obtain strong existence and path-by-path uniqueness for a certain class of fixed point equations, which will take the following form :
\begin{equation} \label{eq:FP}
\forall t \in [0,T], \;\;\;\; X_t = \Gamma \left( x_0 + \int_0^{\cdot} f(s,X_s) ds + W^H \right)(t),
\end{equation}
where $X$ takes values in $\RR^d$, $W^H$ is a ($d$-dimensional) fBm, and $\Gamma$ is a map from $C([0,T],\RR^d)$ to itself. 

\begin{assumption}  \label{asn:Gamma}

Assume that $\Gamma$ is adapted in the sense that for any $t\in [0,T]$, $\Gamma(w)_t$ only depends on $w_{[0,t]}$. Also assume that it is Lipschitz in $p$-variation spaces, namely
\begin{equation} \label{eq:LipGamma}
\forall p \geq 1, \;\; \exists C_P>0 \mbox{ s.t. } \forall y,y ' \in \vvv{p}_{[0,T]}, \;\;\; \left\| \Gamma(y) - \Gamma(y') \right\|_{p-var,[0,T]} \leq C_P \left(  \left\| y - y' \right\|_{p-var,[0,T]} + |y(0) - y'(0)| \right),
\end{equation}

and bounded in H\"older spaces, namely 
\begin{equation} \label{eq:HolderGamma}
\forall \gamma \in (0,1) , \;\; \exists C>0 \mbox{ s.t. } 
\left\| \Gamma(w)
\right\|_{C^\gamma([0,T])} \leq C\left\| w
\right\|_{C^\gamma([0,T])} ,
\end{equation}
and also that letting $C(W):=\Gamma(W) - W$, this has sufficiently regular paths in $q$-variation, at least for fractional Brownian paths :
\begin{equation} \label{eq:asnGammaW}
\exists \; \beta \in (H,1], \eta \in (1/2,1]  \mbox{ with } \eta \beta > \frac{1}{2},  \mbox{ s.t. } \;\;\EE \left[ \exp \left(\left\|C(W^H) \right\|_{1/\beta-var;[0,T]}^{2 \eta} \right)\right] < \infty.
\end{equation}
\end{assumption}

\begin{remark} \label{rmk:TailsH12}
In the case when $H \geq \frac{1}{2}$,  \eqref{eq:asnGammaW} can be replaced by
\begin{equation} \label{eq:asnGammaH12}
\exists \; \beta \in (H,1]  \mbox{ s.t. } \forall x_0 \in \RR^d,\;\left\|C(W^H) \right\|_{1/\beta-var;[0,T]} < \infty \mbox{ a.s.}.
\end{equation}
Indeed, since $\mathcal{H}^{H,T} \subset \vvv{1} \subset \vvv{1/\beta}$ in that case, it follows from a standard Gaussian concentration argument (see Proposition \ref{prop:GaussConc} below) that, under \eqref{eq:LipGamma},  Gaussian tails automatically hold for $\left\|C(W^H) \right\|_{1/\beta-var;[0,T]}$ (as long as it is a.s. finite), so that we do not lose generality by assuming \eqref{eq:asnGammaW}.
\end{remark}

Before stating our theorems, let us define our notions of solutions.

\begin{definition}\label{def:solution}
(1) Let $f \in L^p_t\ccax$ be a distributional drift, and $w \in C([0,T],\RR^d)$ be such that $T^{w} f \in \vvv{r}_t\ccc{1}_x$, for some $r<2$. Given $x_0 \in \RR^d$, we say that $x$ is a solution to

\begin{equation} \label{eq:FPw}
x = \Gamma\left(x_0+ \int_0^{\cdot} f(s,x_s) ds + w \right) 
\end{equation}

if and only if $ \theta := x  -w $ is in $ \vvv{q}$, for some $q< 2$,  and satisfies

\begin{equation}\label{eq:nyoungref}
\theta = \Gamma\left( x_0+  \int_0^{\cdot} (T^{w} f)( ds, \theta_s) + w \right) - w 
\end{equation}
where the integral above is a nonlinear Young integral (which is well-defined given the assumptions on $\theta$ and $T^{w} f$).

(2) We say that path-by-path uniqueness holds for \eqref{eq:FP} if, for $\PP^H$ a.e. fractional Brownian path $w$, there exists at most one solution to \eqref{eq:FPw}.

(3) A strong solution to \eqref{eq:FP} is a stochastic process $(X_t)_{t \geq 0}$, adapted w.r.t. the natural augmented filtration of $\wh$, and such that a.s., $X_t$ is a solution to \eqref{eq:FPw} associated to the path $\wh$.
\end{definition}

Before stating our main results, we first remark that, as in \cite{CG16}, if the driving vector field is sufficiently regular, we can solve directly our equation as a fixed point involving Young integration (in that case only \eqref{eq:LipGamma} is needed as we can simply use the regularizing properties of $W$).

\begin{proposition}
Assume that $\Gamma$ satisfies \eqref{eq:LipGamma} and that $f \in L^p_t C^\alpha_x$, where
\[ \alpha > 2 - \frac{1}{2H} + \frac{1}{pH}. \]
Then path-by-path uniqueness holds for \eqref{eq:FP}.
\end{proposition}

\begin{proof}
By \cite{CG16,GG20}, it holds that, for $\PP^H$-a.e. $w$,  $T^{w}f\in \vvv{q}_t C_{x,loc}^2 $, for some $q< 2$. 

Given two solutions $x,x'$, letting $\theta=x-w$, $\theta'=x'-w$, we can use the linearisation of difference of two solutions, Lemma \ref{lem:linearisation}, which implies existence of a path $A$ such that $\int_0^t T^{w} f(dr, \theta_r)- \int_0^t T^{w} f(dr, \theta'_r) = \int_0^t (\theta_r - \theta'_r) dA_r$ and the last integral is a linear Young integral. It holds that
\begin{align*}
\qvart{ x - x' } & \leq \qvart{ \Gamma\left(w + \int_0^{\cdot} T^{w} f(\theta_r) \right) - \Gamma\left( w + \int_0^{\cdot} T^{w} f(\theta'_r) \right) } \\
& \leq C \qvart{  \int_0^{\cdot} T^{\wh} f(dr,\theta_r) - \int_0^{\cdot} T^{\wh} f(dr,\theta'_r) } = C \qvart{ \int_0^{\cdot} (\theta_r - \theta'_r) dA_r }
\end{align*}

\bluepassage{
We take $x - x' = \theta - \theta'$ and observe that we are exactly in the situation of Lemma \ref{lem:stabest} with $K=0$, along with $x_0 = x_0'$, from which we conclude that $x=x'$ on $[0,T]$.}
%
\end{proof}

We then proceed to our main result, where we improve the requirement on the regularity of $f$. 

\begin{theorem} \label{thm:main}
Assume that $\Gamma$ satisfies {Assumptions} \ref{asn:Gamma}, and $\| f\|_{E} < \infty$, where either :

(1) $H \in (0,1/2)$ and $E =  L_t^p C^\alpha_x$ with $p \in [1,\infty]$ and
\[ \alpha > \left(1 - \frac{1}{2H} + \frac{1}{pH}\right) \vee \left( 1 + \eta\left(1 - \frac{\beta}{H}\right) + \frac{1}{pH} \right), \]

(2) $H \in (1/2, 1)$, and $E = C_t^{\gamma} C^0_x \cap L^\infty_t C^\alpha_x$ with
\[ \alpha > 2 - \frac{\beta}{H}, \;\;\; \gamma > H - \frac{1}{2}, \]
where we recall that $\eta, \beta$ are given by Assumption \ref{eq:asnGammaW}.

Then path-by-path uniqueness holds for \eqref{eq:FP}, and there exists a strong solution. We further have the following stability estimates : given any $M>0$, $m\geq 1$, there exists $C=C(M,m)$ s.t. if $X^1$, $X^2$ are the strong solutions with initial conditions $x^i_0$ and driving vector fields   $f^1$, $f^2$ satisfying $\|f^i\|_E \leq M$, then it holds that 
\begin{equation} \label{eq:stab}
\EE \left[ \sup_{t \in [0,T]} \left|  X^1_t - X^2_t\right|^m \right]^{1/m} \leq C \left ( |x^1_0 - x^2_0| + \left\|f^1 - f^2\right\|_{E'} \right).
\end{equation}
where $E' = L_t^p C^{\alpha-1}_x$ if $H<1/2$, $E'= L_t^\infty C^{\alpha-1}_x$ if $H>1/2$ (with $p,\alpha$ as above).
\end{theorem}

In the next section, we will prove that reflected and perturbed equations satisfy our assumptions with $\beta=1$, $\eta = (H+1/2)\wedge 1$. It is immediate to check that, in that case, the regularity requirements on $f$ in the above Theorem coincide with those given in Theorem \ref{thm:intro}.

\vspace{3mm}

The proof of the theorem will be based on Girsanov theorem and the estimates obtained in section \ref{sec:avgpert}, as in \cite{CG16,GG20}. Namely, we will construct a solution using Girsanov theorem, showing that a drift $T^{\wh + \varphi}f $ satisfies certain continuity and integrability properties. By the equivalence of measures the same is going to hold for $T^{X} f$, where $X$ is a solution. Then due to the fact that $T^{\wh + \varphi} f$ is Lipschitz in space, we will be able to conclude with uniqueness. 
\smallskip

We will proceed by first showing path-by-path uniqueness, and then stability estimates (strong existence will then be an immediate consequence of the latter).
\smallskip

The following lemma shows applicability of Girsanov theorem and will be used repeatedly.

\begin{lemma}\label{lem:girsanov_applicable}
Let $f$ be as in Theorem \ref{thm:main}, and let $h_t = \int_0^t T^{\Gamma(\wh)} f(ds,0)$. Then it holds that

\begin{equation}\label{eq:gauss_cmnorm}
\forall \lambda >0, \;\;\; \EE \left[ \exp\left(\lambda \|h\|^2_{\mathcal{H}^{H,T}} \right)\right] \leq K(\lambda) < \infty,
\end{equation}
where $K$ only depends on $f$ through its norm in $E$.
\end{lemma}

\begin{proof}
We treat separately the cases $H < 1/2$ and $H>1/2$ since they require different arguments. 

\vspace{0.25cm}
\paragraph{$\mathbf{H < 1/2}$}  According to \eqref{eq:cmbesovembedding}, it suffices to show that for some $\epsilon >0$, $\norm{ h }_{W^{H+1/2+\epsilon,2}}$ has Gaussian tails.

We will apply Theorem \ref{thm:continuity}, taking $\phi_t = C(\wh)_t$ and the random control $w_{\phi}(s,t) = \| C(\wh)\|_{1/\beta-var;[s,t]}^{1/\beta}$, and small $\xi>0$. It then holds that  we have for $\delta> H + 1/2 + \epsilon$ and $\gamma_0> \frac{1}{2} + H(1-\eta) $,
$$\norm{ T^{\wh + \phi} f }_{\ccc{0+}}  \ls (t-s)^{\gamma_0} w_{\phi}(s,t)^{\eta \beta} + L(\omega) (t-s)^{\delta}$$ where $L(\omega)$ has Gaussian tails.

This implies that for $\Delta^2 = \{ (s,t) \lvert 0 \leq s \leq T - t, 0 \leq t \leq T \}$
\begin{equation} \label{eq:BndWH}
\| h \|_{W^{H+1/2+\epsilon,2}}^2 \lesssim \int_{\Delta^2} \frac{ \absv{t-s}^{2 \gamma_0} w_{\phi}(s,t)^{2 \eta \beta}}{\absv{t-s}^{2+2H+2\epsilon}}ds dt+ L(\omega)^2 \int_{\Delta^2} \frac{\absv{t-s}^{2\delta}}{\absv{t-s}^{2+2H+2\epsilon}} ds dt.
\end{equation}
The second term is exponentially integrable by Gaussian integrability of $L$ for some $\lambda > 0$ and $\delta > H+1/2$.

For the first one we make use of the following observation, used in Theorem 2.11 in \cite{FP18}. For any control $w$,  it holds that
 \begin{align*}\label{eq:integrcontr}
\int_s^{t-v} w(u,u+v) du = \sum_{i=0}^{M-1} \int_{t_i}^{t_{i+1}} w(u,u+v) du& \leq \sum_{i=1}^M (t_{i+1} - t_i) \sup_{u \in [t_i,t_{i+1}]} w(u,u+v) \\
&\leq v \sum_{i=1}^M w(t_i,t_{i+2})   \leq 2 v w(s,t)
\end{align*}
(where we have taken $\{t_i\}_{i = 0, ..., M+1}$ a subdivision of $[s,t]$ of step size $v$. Applying this observation to $w(s,t) = w_{\phi}(s,t)^{2 \eta \beta}$ which is a control since $\eta \beta > \inv{2}$, we obtain that the first term in \eqref{eq:BndWH} is bounded by
\begin{align*}
 & \int_0^{t-v} \int_s^{t-v} v^{2\gamma_0 - 2H - 2 - \epsilon} w_{\phi}(u,u+v)^{2\eta \beta} du dv  \\ \ls & w_{\phi}(0,T)^{2\eta \beta} \int_0^{t-v} v^{2\gamma_0 - 2H - 1 - \epsilon} dv \ls w_{\phi}(0,T)^{2\eta \beta} T^{2\gamma_0-2H}.
\end{align*}

Finally, this is exponentially integrable by Assumption \eqref{eq:asnGammaW}.

To extend it to all $\lambda > 0$, we decompose $f$ into a smooth and irregular part using Paley-Littlewood blocks :

$$f^1 = \sum_{j: 0 \leq j \leq N} \Delta_j f, \quad f^2 = \sum_{j: j > N} \Delta_j f, $$

we then have for some $C$ and $\eps > 0$:

\begin{equation}\label{eq:split}
\norm{ f^1 }_{L_t^q\ccc{\eps_0}} \leq C \norm{ f }_{L_t^p\ccc{\alpha}_x} 2^{N(-\alpha+\eps_0)}, \quad \quad \norm{ f^2 }_{L_t^q\ccc{\alpha-\eps}} \leq C \norm{ f }_{L^q_t\ccc{\alpha}_x} 2^{-N\eps} .
\end{equation}
\bluepassage{Observe that as $\alpha < 0$ the factor $2^{N(-\alpha)}$ is to be thought of as "big". In the first inequality we used the characterisation of H\"older spaces with Paley-Littlewood blocks, see \cite{BCD11}, Def. 2.68:

\[ \norm{ f^1}_{\ccc{\eps_0}} \leq C 2^{-N(\alpha-\eps_0)} \sup_{j < N} 2^{j\alpha } \norm{ \Delta_j f^1 }_{\ccc{\alpha}}  \]

where $\eps_0 \in (0,\eps)$ and we used $\Delta_k \Delta_j f^1 = 0$ for $k \geq N+1$. The $\sup$ on the right hand-side is then bounded by $\norm{ f }_{\ccc{\alpha}}$. 
}
Let $\bar{\lambda}$ be such that $\EE \exp\left( \bar{\lambda} L^2 \right) < K(\bar{\lambda})$, precisely $\bar{\lambda} \nofn^2 c_0 < 1$ for some constant $c_0 > 1$. For any $\lambda$ we can choose $N(\lambda)$ such that 
\begin{equation}\label{eq:anylambda}
\lambda < c_0^{-1} C^{-2} 2^{2N\eps-2} \bar{\lambda}
\end{equation}

The first one implies that $(t,x) \mapsto \int_0^t f^1_r(\wh_r+\phi_r+x) dr$ is an integral well-defined classically with enough continuity in space. It follows:

\[ \EE \exp\left( \lambda \norm{ \int_0^{\cdot} f^1_r(\wh_r + \phi_r + x ) }_{W^{H+1/2+\eps,2}}^2 \right) < \exp\left( \lambda C^2 \nofn^2 2^{-2N\alpha} \right) \]

For the irregular component we have, we write, repeating steps for the finite $\lambda < \bar{\lambda}$ case:

\[ \EE \exp\left( \lambda \norm{ T^{\wh+\phi} f^2 }^2_{W^{H+1/2+\eps,2}} \right) \leq \EE \exp\left( 2 \lambda w_{\phi}(0,T)^{2\eta\beta} T^{2\gamma_0} \right) + \EE \exp\left(2 \lambda \bar{L}^2 C_T \right) \]

As $w_{\phi}^{\eta \beta}$ was assumed to have Gaussian tails, we can assure finiteness of the first term for every $\lambda$ by taking $\eta' < \eta$ if necessary. For the second term, we use the bound on $L$ obtained in Theorem \ref{thm:continuity} - $\norm{ \bar{L} }_m \leq m^{1/2} C \norm{ f^2 }_{L_t^p\ccc{\alpha-\eps}}$. Then using Taylor expansion of $\exp$, choice of $N$ in \eqref{eq:anylambda} and a bound on $\norm{ f^2 }_{L_t^p\ccc{\alpha-\eps}}$ from \eqref{eq:split} we conclude that the second term is finite as well.

\vspace{0.25cm}
\paragraph{$\mathbf{H > 1/2}$} In this case $f$ is a continuous function, so that it may be evaluated pointwise, and $h$ is simply the well-defined integral
\[ h_t = \int_0^t f(r,x_0+ \Gamma(W)_r) dr. \]
By the Besov Cameron-Martin embedding, it will be enough to show that $t \mapsto f_t(\Gamma(W)_t)$ is an element of $C^{H-1/2+\epsilon,2}$, with (better than) Gaussian tails for its norm. 
However, it follows immediately from \eqref{eq:HolderGamma} and the regularity assumption on $f$, that, for sufficiently small $\epsilon, \kappa>0$,

\begin{equation}\label{eq:fwphi_cont}
\begin{aligned}
\absv{ f(t,\Gamma(W)_t + \phi_t) - f(s,\Gamma(W)_s + \phi_s) } \ls & (t-s)^{H-1/2+\varepsilon} \left ( 1 + \| W\|_{C^{H-\kappa}}^\alpha \right)
\end{aligned}
\end{equation}
and we can conclude since $ \| W\|_{C^{H-\kappa}}$ has Gaussian tails.

\end{proof}

We will also need the following, which is a simple modification of Lemma 7 from \cite{GG20}.

\begin{lemma}\label{lem:equiv}
Let $w,\theta$ be such that $T^{w+\theta}f \in \vvv{q}_t\mathcal{C}_x^1$ and let $\theta \in \vvv{q}_t$ for some $q <2$. Then for any $\tilde{\theta} \in \vvv{q}_t$ it holds:

$$\int_0^{\cdot} T^{w+\theta} f(dr,\tilde{\theta}_r) = \int_0^{\cdot} T^{w} f(dr,\theta_r+\tilde{\theta}_r). $$
\end{lemma}

\begin{proof}
We replace $\absv{ \theta_r - \theta_{t_i} }^{\alpha} = \norm{ \theta }_{q;[t_i,r]}^{\alpha}$ and $\Pi_n = \sup_{ [u,v] \in \PC^n } \norm{ \theta }_{q;[u,v]}$ in the original proof.
\end{proof}

\begin{proof}[Proof of path-by-path uniqueness]

The proof is similar to that of Lemma 10 in \cite{GG20}, which is based on ideas in \cite{CG16}. 

Define the event:
$$U = \{ w : \mbox{ for some }p<2, \; T^{ w } f \in \vvv{p}_t\ccc{1}_x,\ \exists X \in w + \vvv{p}, \ \text{s.t $X$  solves \eqref{eq:FPw}}, T^X f \in \vvv{p}_t\ccc{1}_x \}. $$

\textit{Step 1.} We first want to show that $\PP^H(U) = 1$, where $\PP^H$ is the fBm law on path-space. 

The first part in the definition is satisfied $\PP^H$-a.s., by results of \cite{CG16} - \bluepassage{Theorem 1.1 therein}.

For the second part - start from the measure $\PP^H$, under which $\wh$ is a fractional Brownian motion and let $X = \Gamma(x_0 + \wh)$,
 \[ h_t = \int_0^t T^X f(dr, x_0), \] 
 and note that
\begin{equation}\label{eq:solinit}
X = \Gamma\left( x_0 + \wh + h - h\right) = \Gamma\left( x_0 + \int_0^{\cdot} T^{\Gamma(\wh)} f(dr, x_0) + \wh - h \right) =: \Gamma\left( x_0+h_t + \tilde{W}^H \right)
\end{equation}

Note that the condition on $\alpha$ implies that we can apply Theorem \ref{thm:continuity}, using $\beta$ as in Assumption \ref{eq:asnGammaW} and taking $\eta=1$, for $\xi=1$, to obtain that, $\PP^H$-a.s.,
\[  T^X  f \in \vvv{p}_t\ccc{1}_x \mbox{ for some } p < 2. \]

In addition, by Lemma \ref{lem:girsanov_applicable}, we can apply Girsanov to the drift $h_t = \int_0^t T^X f(dr,x_0)$, so that $\tilde{W}^H = \wh -h$ is a fBm under $\QQ \sim \PP^H$. We then have 
%
%
%
\begin{equation}\label{eq:tildetheta}
{\theta} = X - \tilde{W}^H = \Gamma(x_0+ \wh) -  \wh + h  \in \vvv{p} 
\end{equation}
%

Then using that, by Lemma \ref{lem:equiv}, 
$$\vvv{p}\ccc{1} \ni \int_0^{\cdot}  T^{\tilde{W}^H + {\theta}} f(dr, x_0) = \int_0^{\cdot} T^{\tilde{W}^H} f(dr, x_0 + X - \tilde{W}^H_r) = \int_0^{\cdot} T^X f(dr, x_0)$$ we can rewrite \eqref{eq:solinit} as: 

\begin{equation}\label{eq:xrewritten}
X= \Gamma\left( x_0 + \int_0^{\cdot}  T^{X - \tilde{W}^H + \tilde{W}^H} f(dr, x_0) + \tilde{W}^H \right) = \Gamma\left( x_0 + \int_0^{\cdot}  T^{\tilde{W}^H} f(dr, x_0 + \theta_r ) + \tilde{W}^H \right),
\end{equation}

namely that $X$ is a solution to \eqref{eq:FPw} for the path $\tilde{W}^H$. Equation \eqref{eq:xrewritten} is valid on a set of $\PP^H$-full measure , by equivalence of measures this is also true on $\QQ$, under which $\tilde{W}^H = \wh - h$ is a fBm. Hence the second condition in the definition of $U$ holds for a.e. fBm path, which concludes the proof of the claim.

%
%

\textit{Step 2.} We then show that uniqueness holds for any path $w \in U$. By definition of $U$, there exists a solution $x = \theta + w$, and we assume we are given another one $x' = \theta'+w$, with $\theta' \in\ccdv $. 

Note that it then holds that

\begin{align*}
\int_0^t T^{x' } f(dr, \cdot) - \int_0^t T^{x} f(dr, \cdot) = & \int_0^t T^{x + \theta' - \theta} f(dr, \cdot) - \int_0^t T^{x} f(dr, 0) \\
= & \int_0^t T^{x} f(dr, \theta'_r-\theta_r ) - \int_0^t T^{x} f(dr, 0)
\end{align*}

where we used Lemma \ref{lem:equiv}. We define $A(ds, v) = T^{x} f(ds, v) - T^{x} f(ds, 0)$ \bluepassage{and for any $t \in [0,T]$ there holds $A(t,0) = 0$. $A$ has the same regularity properties as $T^x f$, in particular it is Lipschitz in space. Using \eqref{eq:LipGamma} and $\theta' - \theta = x' - x$ we observe:

\[ \qvart{ x' - x } \leq C  \qvart{ \int_0^{\cdot} A(dr, \theta' - \theta) } \]

which implies that we can directly apply Lemma \ref{lem:stabest} with $K =0$. Since $\absv{\theta_0 - \theta'_0}=0$ it follows that $\norm{ \theta - \theta' }_{q;[0,T]} = 0$. }

\redpassage{
Therefore by Lemma \ref{lem:integrest} we have:
 
 \begin{equation}\label{eq:lipdiffest}
 \qvart{ \int A(dr, v_r) } \ls \left\|  T^{x}  f \right\|_{\vvv{q}_t C_x^1} \qvart{ v} 
 \end{equation}
 
Now we can estimate the difference of two solutions in variation norm and plug in (\ref{eq:lipdiffest}) along with \eqref{eq:LipGamma}:

\[ \qvart{ x- x' } \leq C_P \qvart{ T^x f } \qvart{x - x'} \]

where we used $x'_0 = x_0$ in the variation estimate on nonlinear Young integral \eqref{eq:indefnyoung}. Therefore by Lemma \ref{lem:stabest} applied with $K=0$ we obtain $\qvart{x-x'} = 0$

\begin{equation}\label{eq:oldcontraction}
\begin{aligned}
\qvart{\theta' - \theta} =& \qvart{\Gamma(x_0 + \int_0^{\cdot} T^{X'} f(dr, x_0) + w) - \Gamma(x_0 + \int_0^{\cdot} T^X f(dr, x_0) + w)}\\
= & \qvart{   \int_0^{\cdot} T^{X'} f(dr, x_0) -  \int_0^{\cdot} T^{X} f(dr, x_0) } \\
\ls & \left\|  T^{x}  f \right\|_{\vvv{q}_t C_x^1} \qvart{ \theta - \theta'}
\end{aligned}
\end{equation}

we pick $t_0$ small enough (and taking slightly larger $q$ if necessary) so that (\ref{eq:contraction}) allows us to use the contraction principle, and  we conclude that $\theta \equiv \theta'$ on $[0,t_0]$. We then iterate to extend this to the whole interval $[0,T]$.
}

\end{proof}

%
%
%
%
%
%


We now fix $f^1$, $f^2$ smooth and show that the stability bound from Theorem \ref{thm:main} holds.

\begin{proof}[Proof of stability estimates for smooth $f^1$, $f^2$.]
We will follow the proof of Theorem 3.13 in \cite{GHM21}, the main difference being that we need to bound the difference of solutions in variation norm, with some additional technicalities due to the non-linearity of $\Gamma$. 

 Let $X^i = \Gamma(x_0^i + h^i + W) = x_0^i + h^i + W + K^i$ with $h^i_t = \int_0^t f^i_r(X_r) dr = \int_0^t T^{X^i} f(dr, 0)$. 


For $\zeta \in [0,1]$, let $X^{\zeta} = \zeta X^1 + (1-\zeta) X^2 = x^{\zeta} + h^{\zeta} + W + K^{\zeta}$. Then it holds that

\begin{equation*}
\begin{aligned}
h^1 - h^2 = & \int_0^t f^1_r(X^1_r) dr - \int_0^t f^2_r(X^2_r) dr \\ 
 = & \int_0^t (f^1 - f^2)(X^2_r) dr + \int_0^t \int_0^1 \nabla f^1( X^{\zeta}_r ) d\zeta \cdot (X^1_r - X^2_r) dr \\
 = & \int_0^t (f^1-f^2)(X^2_r) dr + \int_0^t dA_r (X_r^1 - X_r^2) 
\end{aligned}
\end{equation*}

where we took 
\begin{equation*}
A_t = \int_0^t \int_0^1 \nabla f^1(X_r^{\zeta}) d\zeta dr
\end{equation*}.

By the Lipschitz property of $\Gamma$ (see \eqref{eq:LipGamma})  we have, for any $q<2$

\begin{equation*}
\qvart{ X^1 - X^2 } \leq {C} \left( \absv{ X^1_0 - X^2_0} + \qvart{ \int_0^{\cdot} dA_r (X^1_r - X^2_r) } + \qvart{ \int_0^{\cdot} (f^1-f^2)(X^1_r) dr } \right),
\end{equation*}

so that we can apply Lemma \ref{lem:stabest} with \bluepassage{ $K = C \left( \absv{ X^1_0 - X^2_0} + \norm{ \int_0^{\cdot} (f^1- f^2)(X^1_r) dr }_{q;[0,T]} \right) $. } Therefore:

\begin{equation}
\qvart{X^1 - X^2 } \leq  C\exp\left( C \qvart{A}^{q(1+\eps)} \right) \left( \absv{x_0^1 - x_0^2} + \qvart{ \int_0^t (f^1 - f^2)(X^2_r) dr } \right) . \label{eq:stab1}
\end{equation}

%

Let $\QQ^i$, $i=1,2$ be the measure obtained by Girsanov's theorem, such that $W+h^i$ is a fBm.  \redpassage{Then since $h^i = \Gamma(x_0 + W- h^i)$ - ŁM: lhs is a drift, rhs is reflected fBm, so rhs > 0, but the drift not necessarily. so this equality is suspicious. We have $h = T^X = X - \wh = \Gamma( \wh + h ) - \wh$

Adding a blue part to clarify potential confusion, I don't mind removing}.
 \bluepassage{Since \eqref{eq:solinit} $X^i = \Gamma\left( x_0 + \wh + h^i \right)$, by changing measure and using } Lemma \ref{lem:girsanov_applicable} we have that for any $\lambda>0$, 
\[ \EE_{\QQ^i} \left[ \exp(\lambda \|h^i\|_{\mathcal{H}^{H,T}}^2)\right] \leq K(\lambda) < \infty, \]
where $K$ only depends on $\|f^i\|_E$. This then further implies by Theorem \ref{thm:girsanov} (Girsanov) that for any $m \geq 1$,
\begin{equation}
\EE_{\QQ^i} \left[ \left( \frac{ d\PP }{ d\QQ^i } \right)^{m}  \right] + \EE_{\PP} \left[ \left( \frac{ d\QQ^i }{ d\PP } \right)^{m}  \right] \leq C_m,
\end{equation}
where again $C_m$ only depends on $\|f^i\|_E$, a fact that we will use repeatedly in the sequel.

%
%

Let us first bound the moments of the second term in \eqref{eq:stab1}. 

\begin{align*}
\EE_{\PP} \left[ \qvart{ \int_0^{\cdot} (f^1-f^2)(X^2_r) dr }^{m} \right] \leq & \EE_{\QQ^2} \left[ \frac{ d\PP }{ d\QQ^2 }^{2} \right]^{1/2} \EE_{\QQ^2} \left[ \qvart{ \int_0^{\cdot} (f^1-f^2)(X^2_r) dr }^{2m} \right]^{1/2} \\
= &  \EE_{\QQ^2}\left[ \frac{ d\PP }{ d\QQ^2 }^2 \right]^{1/2} \EE_{\PP} \left[ \qvart{ \int_0^{\cdot} T^{\Gamma(x_0+\wh)}(f^2-f^1)(0,dr)}^{2m} \right]^{1/2}.
\end{align*}

We now fix 
\[\frac{1}{1+H(\alpha-1) - \frac{1}{p}} < q < 2\]
(this is possible by the assumptions on $\alpha$).

We then apply Theorem \ref{thm:continuity}, with parameters $\xi>0$ small, $\bar{\alpha}=\alpha-1$,  $\eta,\beta$ as in \eqref{eq:asnGammaW}, (again it follows from the assumption on $\alpha$ that this is possible), which leads to
\[ \EE_{\PP} \left[ \qvart{ \int_0^{\cdot} T^{\Gamma(x_0+\wh)}(f^2-f^1)(0,dr)}^m \right] \lesssim \| f^1-f^2\|_{L^p_t \ccc{\alpha-1}_x}. \]

Now we need to estimate $\EE \exp\left( c \qvart{A}^{q(1+\eps)} \right)$. \bluepassage{Note that by variation embedding $\qvart{A} \leq \norm{A}_{q(1+\eps);[0,t]}$ and since $\eps$ is arbitrarily small we simply reset $q := q(1+\eps) < 2$}.  We have $A= \int_0^1 A^\zeta d\zeta$ where $A^{\zeta}_t = \int_0^t (\nabla f^1)(X^{\zeta}_r) dr$, so that by Fubini and Jensen we have, :

\[\exp(c \qvart{ A }^{q}) \leq \int_0^1 \exp(c \qvart{ A^{\zeta} }^{q})] d\zeta, \]
and it will suffice to bound each of the $\EE\exp(c \qvart{ A^{\zeta} }^{q})$.

We then fix $\zeta \in [0,1]$ and note that
\[ A^{\zeta}_t = \int_0^t (\nabla f^1)(X^{\zeta}_r) dr = \int_0^t T^{\wh + h^\zeta + K^\zeta}(\nabla f^1)(x_0, dr) \]
where $h^\zeta$, $K^\zeta$ are the convex combinations of $h^i$, $K^i$. 

Then we note that  we can find a probability measure $\QQ^\zeta \sim \PP$ s.t. $\wh+h^\zeta$ is a $\QQ^\zeta$ fBm, and $\frac{d\QQ^\zeta}{d\PP}$, $\frac{d\PP}{d\QQ^\zeta}$ have moments of all orders (with bounds only depending on $\|f^i\|_E$).

Indeed, this again follows from Girsanov's theorem and Lemma \ref{lem:girsanov_applicable}, using that
\begin{align*}
\EE[\exp(\lambda \|h^\zeta\|^2)_{\cH}]& \leq  \sum_{i=1}^2  \EE[\exp(4 \lambda \|h^i\|^2_{\cH})]   \\
& \leq \sum_{i=1}^2 \EE[\frac{d\QQ^i}{d\PP}^2]^{1/2}  \EE_{\QQ^i} \left[ \exp(8\lambda \|h^i\|_{\mathcal{H}^{H,T}}^2)\right]^{1/2}.
\end{align*}

We can then apply Theorem \ref{thm:continuity} with the same parameters as above, to obtain that
\begin{align*}
\left\| A^{\zeta} \right\|_{q,[0,T]}^{q} &\lesssim \left\| K^{\zeta} \right\|_{1/\beta, [0,T]}^{q \eta} + (R^\zeta)^{q}  \\
&\lesssim  \left\| K^{1} \right\|_{1/\beta, [0,T]}^{q \eta} +  \left\| K^{2} \right\|_{1/\beta, [0,T]}^{q \eta} +   (R^\zeta)^{q}
\end{align*}
for some $\gamma>0$, where $R^\zeta$ has Gaussian tails under $\QQ^\zeta$. It therefore suffice to show that exponential moments of the three terms above are bounded. This follows from another simple application of Girsanov. Indeed,
\[ \EE[\exp(c  (R^\zeta)^{q})] \leq  \EE_{\QQ^\zeta}\left[\exp(2c  (R^\zeta)^{q})\right]^{1/2} \EE_{\QQ^\zeta}\left[\left(\frac{d\QQ^\zeta}{d\PP}\right)^2\right]^{1/2} < \infty, \]
and similarly, since $K^{i} = \Gamma(x_0+\tilde{W}^i) - (x_0+\tilde{W}^i) = C(x_0+\tilde{W}^i)$ where $\tilde{W}^i$ is a fBm under $\QQ^i$, we can apply the integrability assumption \eqref{eq:asnGammaW} to obtain that
\begin{align*}
 \EE[\exp(c\left\| K^{i} \right\|_{1/\beta, [0,T]}^{q \eta})] \leq\EE_{\QQ^i}\left[\exp(\left\| C(x_0+\tilde{\wh}^i)\right\|_{1/\beta, [0,T]}^{q \eta})\right]^{1/2} \EE_{\QQ^i}\left[\left(\frac{d\QQ^i}{d\PP}\right)^2\right]^{1/2}
 \end{align*}
 is bounded for any $c>0$, with bounds that only depend on the $f^i$ through their norms in $E$.
 
 This concludes the proof of \eqref{eq:stab} for smooth $f^1$, $f^2$.
\end{proof}

We conclude the proof of Theorem \ref{thm:main} by proving strong existence and validity of the stability estimates for general $f$.

\begin{proof}[Proof of strong existence]
Let us quickly recall why the stability estimates yield strong existence (see e.g. \cite{CG16} for similar arguments).

Let $f^n$ be smooth, with $f^n \to f$ in $E$. Let $X^n$ be the associated strong solution to the equation

\begin{equation} \label{eq:fpn}
X^n = \Gamma \left(x_0 + \int_0^{\cdot}T^W f^n(ds, X^n_s - W_s) + W \right).\end{equation}

By the stability estimates, passing to a subsequence if necessary, we can assume that
\[ X^n \to X \mbox{ in }L^m(\Omega, \vvv{q}_t), \]
for some $q<2$, and it only remains to show that $X$ is (a.s.) a solution to the equation with $f$. This follows from the fact that $T^W f^n \to T^W f$ in $\ccc{1/q}_t \ccc{1}_x$, along with continuity of Young integration.

The stability estimate for arbitrary $f^1$, $f^2$ similarly follows by taking smooth approximations $f^{n,1}$, $f^{n,2}$ and passing to the limit.
\end{proof}

\section{Reflected and perturbed equations}\label{sec:RefPert}

In this section, we show that both reflected and perturbed equations, as described in the introduction, fall in the framework of Section \ref{sec:FixedPointEqs}.

\subsection{Reflected equations}

Let us fix a domain $D= \Pi_{i=1}^d [a_i, b_i]  \subset \RR^n$, for some  $-\infty \leq  a_i<b_i \leq +\infty$.

For any continuous function $\phi : [0,T] \to \RR^n$ with $\phi(0) \in D$, there exists a unique pair $(\psi, k)$ of continuous functions such that $\psi=\phi+k$ and
\[\; \forall t \in[0,T], \psi(t) \in D, \]
$k$ has bounded variation and satisfies
\[dk^i(t) = \left(1_{\psi^i(t) = a_i} - 1_{\psi^i(t) = a_i}\right) |dk^i(t)|.\]

The map $\Gamma_D :\phi \mapsto \psi$ is called the Skorokhod map. 
In this section, we will prove the following.

\begin{proposition}
$\Gamma_D$ satisfies Assumption \ref{asn:Gamma} with $\beta=1$ and $\eta = (H+\frac{1}{2}) \wedge 1$.
\end{proposition}

First, note that since 
\[ \Gamma_D(\phi) = \left(\Gamma_{[a_1,b_1]}(\phi^1), \ldots,\Gamma_{[a_d,b_d]}(\phi^d)\right), \]
we can assume w.l.o.g. that $d=1$, and by scaling we can assume that $D$ is either $[0,1]$ or $[0,\infty)$.

It is then classical that $\Gamma$ is Lipschitz continuous on $C^0$ (equipped with supremum norm), see e.g. \cite{KLRS07}. This is readily seen to imply the H\"older bound \eqref{eq:HolderGamma}, by noting that
\[\left | \Gamma_D(\phi(t)) - \Gamma_D(\phi(s)) \right| \leq C \left| \phi(t) - \phi(s)\right| \]
(simply by comparing $\phi$ on $[s,t]$ with the constant path $\phi'\equiv \phi(s)$).

The fact that $\Gamma_D$ is Lipschitz-continuous on $p$-variation spaces was shown by Falkowski and Słomiński \cite{FS15},\cite{FS22} .
%
%

We then proceed to prove the integrability estimate \eqref{eq:asnGammaW}. Note that by Remark \ref{rmk:TailsH12}, it suffices to consider $H< \frac 1 2$. 

If $D=[0,\infty)$, $\Gamma_D(w)-w = -\inf_{s \leq \cdot} (w_s\wedge 0)$, which is non-decreasing on $[0,T]$, so that its $1$-variation norm is bounded by its value at time $T$, and the latter has Gaussian tails by Fernique's estimate. 
 
It remains to consider the case $D=[0,1]$, which is the content of the following proposition. 

\begin{proposition}\label{thm:expintegrbox}
Let $W$ be a fractional Brownian motion of Hurst index $H \in (0,1/2)$, $\Gamma_{[0,1]}(W)$ its reflection on $D=[0,1]$, and $K=\Gamma(W)-W$ its reflection measure. Then for any $T>0$, there exists $\lambda>0$ s.t.
\begin{equation} \label{eq:IntK}
\EE \left[\exp \left( \lambda \norm{K}_{\vvv{1},[0,T]}^{1+2H} \right) \right] < \infty.
\end{equation}
\end{proposition}

In order to prove the theorem, given $W \in \Omega= C([0,T],\RR)$, $\delta >0$, define the successive times 
\[ \tau_0=0, \; \tau_{k+1} = \inf \left\{t \geq \tau_k, \sup_{s \in [\tau_k, t]} |W(t)-W(s)| =\delta\right\} \]
and
\[N_{\delta,T}(W) = \sup \{ k, \tau_k < T \}. \]

The first step is the following lemma.

\begin{lemma}\label{lem:konevarnumber}
 $ \|K\|_{1-var,[0,T]} \leq N_{1,T}(W) +1 $.
\end{lemma}

\begin{proof}
Let $\rho_k = \inf \{t, \norm{K}_{1,[0,t]} \geq k\}$, it suffices to show that $\rho_k \leq \tau_{k}$ with $\tau_k$ defined above for $\delta=1$. 

At time $\rho_k$, $\Gamma(W)$ must be at $0$ or $1$, we assume it is at $0$ (the other case is the same). Now we consider the interval $[\rho_k, \rho_{k+1}]$, on which there are two possibilities :

1) either $\Gamma(W)$ touches $1$ on this interval. Let $t$ be the first time at which this happens, and $s$ the last time before $t$ where $\Gamma(W)$ is at $0$. Then $W(t)-W(s) = \Gamma(W)(t)-\Gamma(W)(s) = 1$. In particular, if $\tau_k \leq \rho_k$, then $\tau_{k+1} \leq \rho_{k+1}$.

2) otherwise, $\Gamma(W)$ does not touch $1$. This means that $dK \geq 0$ and $K(\rho_{k+1}) - K(\rho_k) = 1$, and also $\Gamma(W)(\rho_{k+1}) = \Gamma(W)(\rho_k)= 0$, so that $W(\rho_{k+1}) - W(\rho_k) = -1$, and again $\rho_{k+1} \geq \tau_{k+1}$ if $\rho_k \geq \tau_k$.

\end{proof}

We then use Gaussian concentration of measure in the following form (see e.g. \cite{FH20} Theorem 11.7) 
\begin{proposition} \label{prop:GaussConc}
Given $F: \Omega \to \RR$, assume that for some $G : \Omega \to \RR$  which is $\PP^H$ a.s. finite, and some constant $\sigma>0$, it holds that
\begin{equation}
F(W+h) \leq G(W) + \sigma \norm{ h }_{\mathcal{H}^{H,T}}.
\end{equation}
Then $F(W)$ has Gaussian tails under $\PP^H$, i.e. $\exists \lambda>0, \; \EE^H[ \exp(\lambda F^2)]<\infty$.
\end{proposition}

The lemma below is then a simpler version of \cite{FH20}, Thm 11.13.

\begin{lemma} 
Under $\PP^H$, $N_{1,T}(W)^{1/2+H}$ has Gaussian tails.
\end{lemma}

\begin{proof}
Note that one obviously has
\[N_{1,T}(W+h) \leq N_{1/2,T}(W) + N_{1/2,T}(h).\]
Letting $q:=1/(1/2+H) \in [1,2)$, this implies by subadditivity that 

\[ N_{1,T}(W+h)^{1/q} \leq N_{1/2,T}(W)^{1/q} + N_{1/2,T}(h)^{1/q}.\]
On the other hand, it is readily checked that

\[ N_{1/2}(h)^{1/q} \leq  2 \|h\|_{q-var;[0,T]} \]
and to conclude, we use the fact that for this choice of $q$, $\| h\|_{q-var;[0,T]} \lesssim \| h\|_{\mathcal{H}^{H,T}}$, which was proven in \cite{FGGR16}.

\end{proof}

Proposition \ref{thm:expintegrbox} is then obtained by combining the previous Lemma and Lemma \ref{lem:konevarnumber}.

\begin{remark}
One can check that the exponent $1+2H$ in Proposition \ref{thm:expintegrbox} is sharp. Indeed, let the path $h_N$ be given by
\[ h_N(t) = 2 \cos(2\pi N t/T), t \in [0,T].\]
Then it is easy to see that, letting $K(h_N) = \Gamma_{[0,1]}(h_N) - h_N$, $\|K(h_N)\|_{1-var}$ is of order $N$, and that this is also true for any path $w$ with $\|w-h_N\|_{\infty} \leq \frac{1}{2}$. On the other hand, by the Cameron-Martin theorem, this happens for fBm paths with probability greater than $C\exp\left(- \frac{1}{2} \| h_N\|^2_{\mathcal{H}^{H,T}}\right)$. Since, in addition, for any $\gamma > H+1/2$, it holds that $\|h_N\|_{\mathcal{H}^{H,T}} \geq \|h_N\|_{H^\gamma} \sim N^{-\gamma}$, it follows that \eqref{eq:IntK} does not hold if $1+2H$ is replaced by a higher exponent.
\end{remark}

\subsection{Perturbed equations}

In this section, we fix $\alpha_i,\beta_t, i=1,\ldots, d$, and assume that
\begin{equation} \label{eq:asnAB}
\forall i=1,\ldots, d, \;\; \alpha_i < 1, \; \beta_i < 1, \;\; \rho(\alpha_i,\beta_i):=  \frac{| \alpha_i \beta_i|}{(1-\alpha_i)(1- \beta_i)} < 1.
\end{equation}

We will then show that perturbed equations, written component-wise as
\begin{equation} \label{eq:Pert2}
X_t^i = x_0^i + \int_0^t b^i(s,X_s) ds + W_t^i +  \alpha_i \max_{0 \leq s \leq t} X^i_s +  \beta_i \min_{0 \leq s \leq t} X^i_s.
\end{equation}
fall into the framework of Section \ref{sec:FixedPointEqs}. 

\begin{proposition}
Assume that \eqref{eq:asnAB} holds. Then there exists a map $\tilde{\Gamma}_{\alpha,\beta}$ s.t., for any continuous path $w:[0,T]\to \RR^d$, $f = \tilde{\Gamma}_{\alpha,\beta}(w)$ is the unique function $[0,T] \to \RR^d$ satisfying 
\begin{equation} \label{eq:FPf}
\forall t \geq 0, \forall i=1,\ldots,d, \; f^i(t) = w^i(t) + \alpha_i \max_{0\leq s \leq t} f^i(s) + \beta \min_{0\leq s \leq t} f^i(s). 
\end{equation}
In addition, $\tilde{\Gamma}_{\alpha,\beta}$ satisfies Assumption \ref{asn:Gamma}, with $\beta=\eta =1$.
\end{proposition}

We now prove this proposition. As in the case of the Skorokhod map, it suffices to reason component-wise, so w.l.o.g. we assume that $d=1$, letting $\alpha=\alpha_1$, $\beta = \beta_1$, and we recall that we assume
\begin{equation} \label{eq:asnAB1}
\alpha < 1, \; \beta < 1, \;\; \rho(\alpha,\beta):=  \frac{| \alpha \beta|}{(1-\alpha)(1- \beta)} < 1.
\end{equation}

Let us first note that $f$ solves \eqref{eq:FPf} for $w$ if and only if $f(0) = \frac{1}{1-\alpha-\beta} w(0)$, and $f - f(0)$ solves \eqref{eq:FPf} for $w-w(0)$ (it is easy to check that under \eqref{eq:asnAB1}, $\alpha+\beta<1$). Therefore it is enough to consider the case where $w \in  C_0([0,T],\RR) = \{ w \in C([0,T],\RR), \; w(0)=0 \}$.

It then follows from results of \cite{CPY98} that for any $w \in C_0$, there exists a unique $f := \tilde{\Gamma}_{\alpha,\beta}(w)$ s.t. \eqref{eq:FPf} holds. 

In fact, letting $m^+(w)(t) = \sup_{0 \leq s \leq t} f(s)$ (resp. $m^- = \inf...$), $m^+(w)$ is the unique solution to the fixed point equation
\begin{equation}
m(t) = \frac{1}{1-\alpha} \sup_{s \leq t}\left( w(s) - \frac{\beta}{1-\beta} \sup_{u \leq s} \left( - w(u) - \alpha m(u)\right) \right) =: \phi^+(w,m) (t).
\end{equation}
and similarly, $m^-(w)$ is the unique solution to
\begin{equation}
m(t) =  \frac{1}{1-\beta} \inf_{s \leq t}\left( w(s) + \frac{\alpha}{1-\alpha} \sup_{u \leq s} \left( w(u) +\beta m(u)\right) \right) =: \phi^-(w,m) (t).
\end{equation}
\cite{CPY98} showed that under \eqref{eq:asnAB1}, $\phi^{\pm}$ are a contraction w.r.t. $m$ for the supremum norm. 

As in the previous subsection, this immediately implies the H\"older estimate \eqref{eq:HolderGamma}.

We now show Lipschitz continuity in $p$-variation, which will be a straightforward consequence of the following property of the running supremum, which was obtained by Falkowski and Słominski \cite{FS15}.

\begin{proposition}
Given $w^1, w^2 \in   C_0([0,T],\RR)$, let $y^i(t):= \sup_{0 \leq s \leq t} w^i(s)$, $i=1,2$. Then it holds that, for any $p \geq 1$,
\[ \left\| y^1 - y^2\right\|_{p-var,[0,T]} \leq \left\| w^1 - w^2\right\|_{p-var,[0,T]} \]
\end{proposition}

We then have that \eqref{eq:LipGamma} holds for $\tilde{\Gamma}_{\alpha,\beta}$.

\begin{lemma}
It holds that, for all $w,m,w',w'$ in $ C_0([0,T],\RR)$,
\begin{equation}
\left\| \phi^+(w,m) - \phi^+(w,m') \right\|_{p-var;[0,t]} \leq \rho(\alpha,\beta) \left\| m-m' \right\|_{p-var;[0,t]}
\end{equation}
\begin{equation}
\left\| \phi^+(w,m) - \phi^+(w',m) \right\|_{p-var;[0,t]} \leq  C_{\alpha,\beta} \left\| w-w' \right\|_{p-var;[0,t]}
\end{equation}

As a consequence, assuming \eqref{eq:asnAB1}, we also have, for all $w,w'$ in $ C_0([0,T],\RR)$,
\begin{equation}
\left\| m^+(w) - m^+(w') \right\|_{p-var;[0,t]} \leq \tilde{C}  \left\| w-w' \right\|_{p-var;[0,t]}
\end{equation}
(with $\tilde{C} = \frac{C_{\alpha,\beta}}{1-\rho(\alpha,\beta)}$).

The same result also holds with $\phi^+$, $m^+$ replaced by $\phi^-$, $m^-$.
\end{lemma}

\begin{proof}
Follows immediately by using twice the previous proposition.
\end{proof}

We also have Gaussian tails for $\tilde{\Gamma}(W^H)$.

\begin{lemma}
Assume \eqref{eq:asnAB1} and let $W^H$ be a fBm of Hurst index $H \in (0,1)$. Then for any $T>0$ it holds that $\|m^+(W^H)\|_{1-var;[0,T]}$,  $\|m^-(W^H)\|_{1-var;[0,T]}$ have Gaussian tails.
\end{lemma}

\begin{proof}
By the previous lemma, $m^+$ is Lipschitz in $p$-variation, so that by Gaussian concentration (Proposition \ref{prop:GaussConc}) and Cameron-Martin variation embedding \eqref{eq:jainmonrad} , $\|m^+(W^H)\|_{p-var;[0,T]}$ has Gaussian tails for  $p = (H+1/2)^{-1} \vee 1$. But since $m^+$ is increasing, all its $p$-variation norms are equivalent, so that this also hold for $p=1$.
\end{proof}

%

\appendix
\section{Appendix}
\subsection{Stochastic sewing}

\begin{lemma}\label{lem:existence}
Let $A$ be a germ decomposing as Lemma \ref{lem:asutdecomp}. Then for every $x \in \RR^d$ there exists a path $\AAC(x): [0,T]\mapsto \RR^d$  such that:

$$\forall (x,y) \in \RR^d, (s,t) \in \Delta^2\ \absv{ \AAC_{st}(x) - \AAC_{st}(y) - A_{st}(x) + A_{st}(y) } \leq \absv{x-y}^{\xi} (t-s)^{\gamma} \omega(s,t)^{\eta \beta} + B^{x,y}_{st}$$

where $B^{x,y}_{st} : \Delta^2 \times \Omega \mapsto \RR$ is a random 2-parameter map with $\norm{ B^{x,y}_{st} }_{m}  \ls C_m \absv{x-y}^{\xi} (t-s)^{1/2+\epsilon}$, $\gamma + \eta \beta>1$ and $C_m \sim m^{1/2}$ for $m \to \infty$.
\end{lemma}

\begin{proof}
Denote $A^{x,y}_{st} = A^x_{st} - A^y_{st}$ for conciseness. We continue with the decomposition of $\delta_u A_{st}^{x,y}$ obtained in Lemma \ref{lem:asutdecomp} and follow the footsteps of \cite{ABLM21}, Theorem 4.7. We will denote intervals of dyadic partition with:

\begin{equation}\label{eq:dyadpts}
t_0^k = s, \quad t_{2^k-1}^k = t, \quad t_i^k = t_0 + i 2^{-k} (t-s), \quad \bluepassage{u_i^k = \inv{2} \left( t_i^k + t_{i+1}^k \right)}.
\end{equation}

Define for every $(s,t) \in \Delta_2$:
$$A^{(x,y),k}_{st} = \sum_{i=0}^{2^k -1} A^{x,y}_{t_i^k t_{i+1}^k}, $$

and then:

$$A^{(x,y),k+1}_{st} - A^{(x,y),k}_{st} = \sum_{i=0}^{2^k-1} \delta A^{x,y}_{t^k_i,u^k_i,t^k_{i+1}} = I^k_{st} + J^k_{st}, $$

and denote $$I^{(x,y),k}_{st} = \sum_{i=0}^{2^k-1} \delta_{u^k_i} A^{x,y}_{t^k_i, t^k_{i+1}} \quad J^k_{st} = \sum_{i=0}^{2^k-1} M^{(x,y),k}_{t^k_i, t^k_{i+1}},$$

where $M^{(x,y),k}$ is defined precisely below in \eqref{eq:partitionmgale}. We will suppress the notation with $(x,y)$ if there is no risk of confusion. By the assumptions and Holder inequality we obtain:

\begin{equation}\label{eq:contr}
\absvv{I^k_{st}} \leq  C_1 \absv{x-y}^{\xi} 2^{-k(\gamma + \eta \beta - 1)} (t-s)^{\gamma}  \omega(s,t)^{\eta \beta}
\end{equation}

Note that in the above we disregard the fact that decomposition $\delta_u A_{st}$ yields $\omega(t^k_i,u^k_i)$ rather than $\omega(t^k_i,t^k_{i+1})$ - we simply replace the former by the latter and H\"older inequality still works. In order to see how to deal with $J^k_{st}$, let us first write down see the role which it plays in the decomposition.

\[ \absv{ A^{\infty}_{st} - A^0_{st} } \leq \sum_{k \geq 0} \absvv{A^{k+1}_{st} - A^k_{st} }  \leq \absv{x-y}^{\xi} (t-s)^{\gamma} \omega(s,t)^{\eta \beta} \sum_{k \geq 0} 2^{-k(\gamma + \eta \beta - 1)} + \sum_{k \geq 0 }  \absv{ J^k_{st} } \]

so we define $B^{x,y}_{st} = \sum_{ k \geq 0 } \absv{J^k_{st}(x) - J^k_{st}(y)} $

We aim at obtaining $\norm{ J^k_{st}}_{m} \ls \nofn \absv{x-y}^{\xi} C_m 2^{-k \epsilon} (t-s)^{1/2+\epsilon}$ \bluepassage{with $C_m$ as in the statement of the Lemma}, which will then allow us to have finite constant in the moment bound $B$ by the geometric series argument.

Similarly to the proof of Lemma \ref{lem:asutdecomp}, set again $\bar{f}_r(\cdot) = f_r(\cdot + x) - f_r(\cdot + y)$. Then

\begin{equation}\label{eq:partitionmgale}
M_{t_i^k, t_{i+1}^k} = \int_{t_i^k}^{u^k_i} \int_{u^k_i}^{t^k_{i+1}} (P_{zr} \nabla \bar{f}_r)(W^H_r + \phi_s) (z-r)^{H-1/2} dr  dW_z = \int_{t_i^k}^{u^k_i} h^{(x,y),k}_z dW_z
\end{equation}

which results in:

\begin{equation}\label{eq:jstk}
J_{st}^k = \sum_{i=0}^{2^k-1} \int_{t_i^k}^{u^k_i} h_z dW_z = \int_s^t \indic{ z \in [t_i^k, u^k_i]} h_z dW_z 
\end{equation}

then $z \mapsto \indic{ z \in [t_i^k, u^k_i]} h_z $ is adapted and progressively measurable, so BDG is applicable.

Denote  $g(z)_{ik} = \absvv{ \int_{u^k_i}^{t^k_{i+1}} (P_{zr} \nabla \bar{f}_r)(W^H_r + \phi_s) (z-r)^{H-1/2}dr }$ and, using Proposition \ref{prop:hoelderdiff} on $\bar{f}_r$ again in the last line we bound this term:

\begin{align*}
g(z)_{ik} & \leq \int_{u^k_i}^{t^k_{i+1}} \absvv{ (P_{zr} \nabla \bar{f}_r)(W^H_r + \phi_s) (z-r)^{H-1/2} } dr \\
\leq & \int_{u^k_i}^{t^k_{i+1}} (z-r)^{-H(\xi -\alpha) - 1/2} \norm{ \bar{f}_r }_{\ccc{\alpha-\xi}} dr \\
\leq & \absv{x-y}^{\xi} \nofn \big( (t^k_{i+1} - z)^{1/2 - H(\xi-\alpha) - \inv{p}} - (u^k_i - z)^{1/2-H(\xi-\alpha) - \inv{p}} \big)
\end{align*}

\redpassage{
then we observe that for $a > b > 0$ and $\beta, \theta \in (0,1)$ we can write:

$$a^{\beta} - b^{\beta} \leq a^{\beta} \land (a-b)^{\beta} \leq a^{\beta \theta} (a-b)^{\beta (1-\theta)} $$

which finally gives us:
\begin{equation}
g(z)_{ik} \ls \nofn (t^k_{i+1} - u^k_{i+1})^{\theta (1/2 - H(1+\epsilon - \alpha) )} (t^k_{i+1} - z)^{(1-\theta)(1/2 - H(1+\epsilon - \alpha) ) }
\end{equation}

denote $\zeta = 1/2 - H(1+\epsilon - \alpha)$ for brevity and pick $\theta \in (0,1)$ }

We then apply BDG inequality with the optimal constant \bluepassage{$C_p$, which satisfies $C_p \sim p^{p/2}$ as $p \to \infty$}:

\begin{equation}\label{eq:dyadicmgale}
\norm{ J_{st}^k }_{p}^p \ls C_p \left[ \sum_{i=0}^{2^k} \int_{t^k_i}^{u^k_i} (g(z)_{ik})^2 dz \right]^{p/2} 
\end{equation}

\bluepassage{and estimate the integral by taking $\sup$ of the integrand:

\[ \int_{t^k_i}^{u^k_i} \left( g(z)_{ik} \right)^2 dz \ls \absv{x-y}^{\xi} \nofn \absv{ t^k_i - u^k_i } \absv{ t^k_{i+1} - t^k_i }^{1+2H(\alpha-\xi) -\frac{2}{p}}  \]

which is bounded with $2^{-k(1+\epsilon)} \absv{ t-s}^{1+\eps}$ using the condition on $\alpha$ and the choice of the sequence $(t^k)_{k \geq 1},(u^k)_{k \geq 1}$. This way \eqref{eq:dyadicmgale} is bounded uniformly in $k$, which finishes the proof.
}

\redpassage{
\begin{equation}\label{eq:dyadicmgalecont}
\begin{aligned}
\sum_{i=0}^{2^k} \int_{t^k_i}^{u^k_i} g(z)_{ik}^2 dz & \ls \nofn^2 \sum_{i=0}^{2^k} (t^k_{i+1} - u^k_{i+1})^{2\theta \zeta} \int_{t^k_i}^{u^k_i} (t^k_{i+1} - z)^{2(1-\theta) \zeta  } dz \\ \ls & \sum_{i=0}^{2^k} \nofn^2 (t^k_{i+1} - u^k_i)^{2 \theta \zeta} (t^k_{i+1} - t_i^k)^{2(1-\theta)\zeta} (u^k_i - t^k_i) \ls \nofn 2^k 2^{-k(1+2\zeta)} (t-s)^{1+2\zeta}
\end{aligned}
\end{equation}

We end up with:

$$\norm{ J^k_{st} }_{p} \ls \absv{x-y}^{\xi} p^{1/2} \nofn (t-s)^{1/2+\zeta} 2^{-k  \zeta} $$

for some $\zeta > 0$, as required}
\end{proof}

\begin{lemma}\label{lem:existall}
$\AAC$ as above exists as well as a limit of arbitrary partition $\PC$.
\end{lemma}

\begin{proof}

The idea of the proof is quite standard and follows Lemma 2.15 in \cite{Le20}. 
~\paragraph{Dyadic allocation}
~\\
 To obtain convergence for arbitrary partition, we make use of dyadic decomposition as in Lemma 2.14 in \cite{Le20}. We have then, for any partition $\PC = \{ t_0 \leq t_1 \leq ... \leq t_N \}$:

\begin{equation}\label{eq:dyadicalloc}
\sum_{i=0}^N A_{t_i,t_{i+1}} - A_{t_0,t_N} = \sum_{n \geq 0} \sum_{i=0}^{2^n} R_i^n 
\end{equation}

with $R^n_i = - \delta_{s_2^{n,i}} A_{s^{n,i}_1,s^{n,i}_3} - \delta_{s_3^{n,i}} A_{s^{n,i}_1,s^{n,i}_r}$, where all $s_j^{n,i}$ are picked to lie inside the dyadic subinterval on level $n$, that is inside $[t^n_i, t^n_{i+1}]$. By decomposition in Lemma \ref{lem:asutdecomp} and estimates above in the dyadic existence Lemma \ref{lem:existence}. Since all $s_j^{n,i}$ are inside the dyadic subinterval, it means that $R^{n,i}$ is again a term decomposed into $I^n + J^n$, that is a $I^n$ which is a term controlled by $w(s,t)^{1+\delta}, \delta > 0$, $w$ control function and $J^k$ which is a martingale term. 

~\paragraph{Cauchy sequence}
We will sketch an argument why a $\AAC$ can be understood as a limit of Riemann sums in $L^p$. Denote $w(s,t)^{1+\eps} = (t-s)^{\gamma} w_{\phi}(s,t)^{\eta \beta}$

By taking two partitions $\PC, \PC'$ and considering their refinement, we can write for any difference of partitions $A^{\PC}_{st} - A^{\PC'}_{st} = \sum_k^{N''} Z_k$, where $Z_k$ is defined as follows:

\[ Z_k = \sum_{j; t_k < u_j < t_{k+1}} A_{u_j,u_{j+1}} - A_{t_k,t_{k+1}} \]

For any $Z_k$ we can apply \eqref{eq:dyadicalloc}, with decomposition of $R^n_i$ discussed above. Therefore $Z_k = I_k + J_k$, where $I_k$ is a term with a regular random control and $J_k$ is a martingale term. Therefore $\sum_k Z_k = \sum_k I_k + \sum_k J_k$. Repeating the same reasoning as in the proof of existence

$$\absv{ I_k } \leq \sum_{n \geq 0} \sum_{i=0}^{2^n} w(t^n_i,t^n_{i+1})^{1+\eps} \leq C w(t_k, t_{k+1})^{1+\eps}$$

Then adding them together we obtain:

$$\absv{ \sum_k I_k } \leq \sum_k (t_{k+1}-t_k)^{\gamma} w(t_k,t_{k+1})^{\eta \beta} \leq w_{\pi} w(s,t)^{1+\eps/2} \quad w_{\pi} = \sup_{[u,v] \in \PC} w(s,t)^{\eps/2}$$

For the martingale part, first we observe that $J_k$ is a martingale difference itself, that is in itself a sum of martingale differences indexed on dyadics. For each $J_k$ we get, using moment bounds on martingale difference in Lemma \ref{lem:asutdecomp} (again, performing the same computation as in the proof of existence):

$$\norm{ J_k }_m \leq (t_{k+1} - t_k)^{1/2+\eps}$$

by martingale differences of $J_k$ we get that \[ \norm{ \sum J_k }_m \leq \pi (t-s)^{1/2+\eps/2} \quad \pi := \sup_{[u,v] \in \PC} (v-u)^{\eps/2} \]

Clearly, for $\absv{\PC} \to 0$ we have $\pi + w_{\pi} \to 0$, so that:

\[ \norm{ A^{\PC}  - A^{\PC'} }_m \leq \pi + w_{\pi} \to 0 \]

so $A^{\PC}$ is Cauchy in $L^m, m \geq 2$.
\end{proof}
~\paragraph{Uniqueness}
\begin{lemma}\label{lem:sewuniq}
Let $R$ be a real valued process, for which it holds:

$$\forall\ (s,t) \in \Delta^2\;\;  R_{st} \leq I_{st} + M_{st}$$

such that $\absv{ I_{st} } \leq w(s,t)^{1+\eps}$ for $w: \Delta_2 \times \Omega \mapsto \RR_+$ is a continuous random control function  and $M_{st}$ satisfies $\EE_s M_{st}=0$, with moment bounds $\norm{ M_{st} }_2 \ls (t-s)^{1/2+\eps}$ for every pair $t>s$.

Then it holds that $R_{st} = 0$ a.s.

\end{lemma}

\begin{proof}
Let $\PC$ be an arbitrary partition of $[s,t]$ with $N$ being a number of subintervals,  and let us write, due to the fact that $t \mapsto R_t$ is additive:

$$R_{st} = \sum_{i=0}^{N} R_{t_i,t_{i+1}} \leq \sum_{i=0}^{N} I_{t_i,t_{i+1}} + M_{t_i,t_{i+1}} $$

Since $I_{st}$ is bounded by a random control function, which does not necessarily has to keep this property under taking expectation, we first apply summation. We first want to show convergence to zero of the first term and only then apply expectation to control the second term. This is a bit different from a usual scheme of proof from \cite{Le20}, where both terms are bounded using their respective moment bounds. 

Then we have:

\begin{equation}\label{eq:diff}
\absv{ R_{st} } \leq \absv{ \sum_{i=0}^{N} I_{t_i,t_{i+1}} } + \absv{ \sum_{i=0}^{N} M_{t_i,t_{i+1}} } = A^N + B^N
\end{equation}

For the first term we apply triangle inequality and obtain, using superadditivity:

$$\absv{ \sum_{i=0}^{N} I_{t_i,t_{i+1}} } \leq \sum_{i=0}^{N} \absv{ I_{t_i,t_{i+1}} } \leq \sum_{i=0}^{N} w(t_i,t_{i+1})^{1+\eps} \leq w(s,t) \sup_{[t_i,t_{i+1}] \in \PC} w(t_i,t_{i+1})  $$

Since size of the partition converges to zero we have $\sup_{[t_i,t_{i+1}] \in \PC} w(t_i,t_{i+1}) \to 0$. 

To consider the second term, we write
%
%

$$\EE \absv{ B^n }^2 \leq \EE \absv{ \sum_{i=0}^{N} M_{t_i,t_{i+1}} }^2 \leq  \sum_{i=0}^{N} \EE \absv{ M_{t_i,t_{i+1}} }^2$$

 This results in $\EE \absv{ M_{t_i,t_{i+1}} }^2 \ls (t_{i+1} - t_i)^{1 + 2 \eps}$.  in the end we obtain:

$$\EE \absv{ B^n }^2 \leq \sum_{[t_i,t_{i+1}] \in \PC} (t_{i+1} - t_i)^{1+2\eps} \leq (t-s) \sup_{[t_i,t_{i+1}] \in \PC} (t_{i+1}-t_i)^{2\eps} $$

and ultimately for both $A^n$ and $B^n$ we have convergence to zero as $n \to \infty$. It is enough to conclude that for all $(s,t) \in \Delta^2$ we have $R_{st} = 0$ a.s.
\end{proof}

\subsection{Chaining}\label{subsect:chain}

The following lemma will be very useful in the chaining arguments below.

\begin{lemma}\label{lem:aacpathwise}
Let $\AAC$ be a $\RR$-valued path such that:

\begin{equation}\label{eq:aacbound}
\absv{ \AAC_{st} } \leq  w(s,t) + B_{st} 
\end{equation}

with $w$ being a control function and $B_{st}$ a 2-parameter $\RR$-valued function.

Then there exist a map $Y$ such that 
\begin{equation}\label{eq:triangle2}
\forall\ s<u<t, \;\; \absv{ Y_{st} } \leq \absv{ Y_{su} } + \absv{ Y_{ut} } , \quad \absv{ Y_{st} } \leq \absv{ B_{st} }, 
\end{equation} and

\begin{equation}
\forall\ s<t, \;\; \absv{ \AAC_{st} } \leq w(s,t)+ Y_{st} .
\end{equation}

\end{lemma}

\begin{proof}

Let us define $Y_{st} = ( \absv{ \AAC_{st} } - w(s,t) )_+$. Clearly 
\begin{equation}\label{eq:yupb}
\absv{ Y_{st} } \leq \absv{ w(s,t) +  B_{st} -  w(s,t) } \leq B_{st} ,
\end{equation}
so that $Y$ enjoys the same moment bounds as $B$. Also:

\begin{equation}\label{eq:newaacbound}
\absv{ \AAC_{st} } = \absv{ \AAC_{st} } -  w(s,t) +  w(s,t) \leq Y_{st} +  w(s,t) .
\end{equation}

Now we have to show that $Y$ satisfies \eqref{eq:triangle2}. 

In the first case, assume that for every of the three pairs involved there is $\absv{ \AAC } > w$. Then by the superadditivity of $w$ and the fact that $\AAC$ is a path, which makes triangle inequality automatically applicable, we obtain:

$$Y_{st} = \absv{ \AAC_{st} } - w(s,t) \leq \absv{ \AAC_{su} + \AAC_{ut} } - w(s,u) - w(u,t) \leq \absv{ \AAC_{su} } + \absv{ \AAC_{ut} } - w(s,u) - w(u,t) = Y_{su} + Y_{ut} $$

We can write the same inequalities in other cases, i.e. when for either of $(s,u), (u,t)$ we have $\absv{ \delta \AAC } < w$. Then the desired inequality is clearly satisfied as $\delta Y = 0$. 

\end{proof}

\begin{lemma}\label{lem:chain}
Let $U^{x,y}_{st}$ be a \bluepassage{random map,  a.s. continuous in all its arguments and } satisfying the triangle inequality for each $(x,y)$ and every $s <v < t$ : $\absv{ U^{x,y}_{st} } \leq \absv{ U^{x,y}_{sv} } + \absv{ U^{x,y}_{vt} }$. Moreover, assume that for any pair $(s,t)$ there exist a function $x \mapsto J^x_{st}$ such that $\absv{ U^{x,y}_{st} } \leq \absv{ J^x_{st} - J^y_{st} }$, and that we have the following moments bounds for $J^x$:

$$\norm{ J^x_{st} - J^y_{st} }_m \leq D \sqrt{m} (x-y)^{\delta} (t-s)^{\alpha} $$

It then follows that almost surely, for any $\bar{\delta} < \delta, \bar{\alpha} < \alpha, R< \infty$:

\begin{equation}\label{eq:jointhoeld}
L = \sup_{ |x|,|y| \leq R,x \neq y, t \neq s} \frac{ U^{x,y}_{st} }{(x-y)^{\bar{\delta}} (t-s)^{\bar{\alpha}}} < \infty 
\end{equation}

and $L$ has Gaussian tails for $\lambda > 0$ small enough.
\begin{equation}\label{eq:hoelderchaingauss}
\EE \exp\left( \lambda \left( \frac{ L }{D} \right)^2 \right) < \infty
\end{equation}
\end{lemma}

\begin{proof}
Analogously to \cite{CG16}, Lemma 3.1, define:

\begin{equation}\label{eq:rmu}
R_{\mu^{z,w} c} = \sum_{m \geq 0} \sum_{k=0}^{2^k} 2^{-2m} \exp\left( \mu^{z,w} c 2^{2 m \bar{\alpha}} \absv{ U^{z,w}_{k 2^{-m},(k+1) 2^{-m} }}^2 \right)
\end{equation}

Fixing $z,w$ it can be concluded directly from there (or Lemma 18 in \cite{GG20}) that for every $(s,t)$ holds:

$$ \absv{ U^{z,w}_{st}} \leq (t-s)^{\bar{\alpha}} \left( \frac{ \ln R_{\mu^{z,w}} }{\mu^{z,w}} \right)^{1/2} $$

By using Taylor expansion, using moment bounds on $U^{z,w}_{k 2^{-m},(k+1) 2^{-m}}$ and recalling Stirling inequality one can easily see that for 
\begin{equation}\label{eq:mupick}
\mu_0^{z,w} = \inv{4} (z-w)^{-2\bar{\delta}} c^{-1} D^{-2}
\end{equation} we have $\EE R_{\mu^{z,w} c} < C$, with $C$ some constant that is independent from $z,w,c,\delta,D$.

\begin{equation}\label{eq:asbd}
\begin{aligned}
\absv{ U^{x,y}_{st} }^p \leq & (t-s)^{\bar{\alpha} p} \left( \frac{ \ln R_{\mu^{x,y}} }{\mu^{x,y}} \right)^{p/2} \\
\leq & D^p (x-y)^{\bar{\delta} p} (t-s)^{\bar{\alpha} p} \left( \sup_{z \neq w} \frac{ \ln R_{\mu^{z,w}} }{ (z-w)^{2 \delta } \mu^{z,w}} \right)^{p/2} \\
\end{aligned}
\end{equation}

Using definition of $R_{\mu}$ we see that \eqref{eq:asbd} can be bounded by:

\begin{equation}\label{eq:uxtsyp}
\absv{ U^{x,y}_{st} }^p \leq D (x-y)^{\bar{\delta} p} (t-s)^{\bar{\alpha} p} \left( \sum_{m \geq 0} \sum_{k=0}^{2^m-1} 2^{-2m} \exp\left( c_0 D^{-2} \sup_{z\neq w} \frac{ \absv{ U^{z,w}_{k2^{-m},(k+1)2^{-m}} }^2 }{(z-w)^{2\bar{\delta}}} \right)  \right)^{p/2} 
\end{equation}

Before we proceed, let us make a remark how proportionality constant grows with $p$. In particular, by Jensen inequality it can be seen that $R_{\mu}^p \leq R_{\mu p}$ for $p \geq 1$. Therefore a constant $c$ in \eqref{eq:rmu} is now proportional to $p$. By \eqref{eq:asbd} we see that $L^p$ as in \eqref{eq:jointhoeld} is bounded by $\mu^{-1/2}$, so that we have $L^p \leq p^{p/2} C(\omega)$. 

We then use the fact that for any fixed $(s,t)$ we have ${U^{z,w}_{st} } \leq \absv{ J^x_{st} - J^y_{st}}$ by assumption. Then for any fixed $(s,t)$ a map $x \mapsto J^x_{st}$ is a function, therefore by using Kolmogorov continuity theorem (see, for instance \cite{Tal14}, Appendix A.2 for a version of this theorem for functions with arguments in $\RR^d$) we conclude that for every pair $(k,m)$ we have $\sup_{z \neq w} \frac{ \absv{ J^x_{k2^{-m},(k+1)2^{-m}} - J^y_{k2^{-m},(k+1)2^{-m}} }^2 }{\absv{z-w}^{2\bar{\delta}}} < \infty$ almost surely. What remains is to show that $L(\omega)$ is almost surely finite and to that end we will show $\EE L < \infty$. By following the proof of chaining lemma we see that we need to take a high enough moment of H\"older constant in order to obtain correct continuity exponent. Moreover, denote $C_{J,s,t}(\omega) = \sup_{x \neq y} \frac{ \absv{ J_{st}^x - J_{st}^y} }{(x-y)^{ \bar{\delta}}}$ and remark that $\EE C_{J,s,t}^q \ls (t-s)^{\alpha q} q^{q/2} D$, that is moments of H\"older constants are bounded by the same constant as the one in moment bounds, where $q$ is aforementioned high moment. 

Then using Taylor expansion on exponential inside $R_{\mu p}$, a moment bound on $C_{J,k2^{-m},(k+1)2^{-m}}$ we see that $\EE L^p < p^{p/2} D^p$. Therefore it is almost surely finite. Again by expanding with Taylor we see that $\EE \exp\left( \lambda L^2 \right) < \infty$ and $\lambda < D^{-2}$, hence the conclusion.

\end{proof}


\subsection{Global estimate on linear Young inequality in variation norm}\label{subsec:variationyoung}

We give here the proof of Lemma \ref{lem:stabest}. \bluepassage{ Recall that we assume that, for each $t\geq 0$,

\[ \qvart{ Y} \leq  \qvart{ \int_0^{\cdot} A(dr, Y_r) } + K \leq C_q \qvartCx{A} \left( |Y_0| + \qvart{Y} \right)+ K \]


Clearly, for all $t < t_0$ such that $\norm{A}_{q;[0,t]} < 1/(2C_q)$ we have $\qvart{Y} < C\left(  \absv{ Y_0} + K \right)$ for some $C \geq 2$.  Let us subdivide $[0,T]$ in intervals $[t_i,t_{i+1}]$ such that $\dqvar{A}{i} := \norm{ \delta A }_{q;[t_i,t_{i+1}]} < 1/(2C_q)$ - that is, the seminorm part of the variation norm. We will show by induction that for constant $C_R > 1$ coming from \eqref{eq:indefnyoung}
\begin{equation}\label{eq:induction}
\norm{ Y }_{q;[0,t_N]} \leq 2 C_R g(N) \left(  \absv{ Y_0} + K \right)
\end{equation} where $g: \NN \mapsto \NN$ is an increasing function, which we will identify. It holds that

\begin{align*}
\norm{ Y}_{q;[0,t_N]}  \leq &  \norm{ \int_0^{\cdot} A(dr, Y_r) }_{q;[0,t_{N-1}]} + \norm{ \int_0^{\cdot} A(dr, Y_r) }_{q;[t_{N-1},t_N]} + K \\
\leq & C_q C_R \left( \absv{ Y_0} + \norm{ Y}_{q;[0,t_{N-1}]} \right) \norm{ A }_{q;[0,t_{N-1}]} + C_q C_R \left( \absv{Y_{t_{N-1}} } + \norm{ Y}_{q;[t_{N-1},t_N]} \right) \norm{ A }_{q;[t_{N-1},t_N]} + K\\
\end{align*}

Using $\norm{ Y}_{q;[t_{N-1},t_N]} < \norm{ Y}_{q;[0,t_N]}$, $ \norm{ A }_{q;[t_{N-1},t_N]} < 1/(2C_q)$ and \eqref{eq:induction}

\begin{align*}
\norm{ Y}_{q;[0,t_N]} < & 2 C_q C_R \left( \absv{ Y_0} + \norm{ Y}_{q;[0,t_{N-1}]} \right) \norm{ A }_{q;[0,t_{N-1}]} + 2K \\
< & 2 C_q C_R \left( \absv{Y_0} + K \right) \left( (N-1) + 2C g(N-1) (N-1) + 1 \right)
\end{align*}

We now have to identify $g$ such that $C_q \left( N + 2C g(N-1) (N-1) \right) < g(N)$. Take $g(N) = C_g e^{C_u N^{1+\delta}}$ for some $C_g, C_u > 1$ to be identified and arbitrarily small positive $\delta$. Set $\bar{C}$ a constant depending on $C_q, C$. We then need:

\[ C_q N e^{-C_u N^{1+\delta}} + 2 \bar{C} N < C_g e^{C_u \left( N^{1+\delta} - (N-1)^{1+\delta} \right) } \]

By Taylor we have $N^{1+\delta} - (N-1)^{1+\delta} > (N-1)^{\delta}$. We pick $C_u$ such that:

\[ \sup_{n \geq 1} C_q n e^{-C_u n^{1+\delta}} \leq C_g \quad  \inf_{n \geq 1}\left( e^{C_u n^{\delta}} - 1 - 2C n \right) > 0 \]

Considering $C_u \to \infty$ we see that such a choice always exists. To conclude we need to estimate $N$. We can estimate sum of splits by observing that by H\"older inequality and superadditivity of $\dqvar{A}{j}^q$ we have :

\[ \sum_{i=0}^M \dqvar{A}{i} \leq \left( \sum_{i=0}^N 1 \right)^{1/q^*} \left( \sum_{i=0}^N \dqvar{A}{i}^q \right)^{1/q} \leq N^{1-1/q} \qvart{A} \]

combined with (this is where continuity of $t \mapsto \qvart{A}$ is important) $N \inv{2C_q} =  \sum_{i=0}^M \dqvar{A}{i}  \leq N^{1-1/q} \qvart{A}$ we have: 
\begin{equation}\label{eq:sumtotqvar}
N \leq (2C_q)^q \qvart{A}^q
\end{equation}

The conclusion follows by considering $g(N)$ in \eqref{eq:induction}

%
}

\subsection{Useful estimates on H\"older functions}

\begin{proposition}\label{prop:hoelderdiff}. Let $f \in \ccc{\zeta + \eta}$ for $\zeta, \eta \in (0,1]$. Then:

\begin{equation}
\norm{ f(x+\cdot) - f(\cdot) }_{\ccc{\eta}} \ls \norm{ f}_{\ccc{\eta+\zeta}} \absv{x}^{\zeta} 
\end{equation}

and as a result

$$ \absv{ f(x - a) - f(x - b) - f(y - a) + f(y - b) } \ls \norm{ f}_{\ccc{\zeta+\eta}} \absv{ x- y}^{\zeta} \absv{a-b}^{\eta} $$
\end{proposition}

\begin{proof}

The first bound is \cite{BDG19} Proposition 3.8, shown there in Appendix A.

The second one follows by directly applying the first one - define $g_h(s) = f(h+s) - f(s)$. Then:

\begin{align*}
\absv{  f(x - a) - f(x - b) - f(y - a) + f(y - b) } & = \absv{ g_{x-y}(y-a) - g_{x-y}(y-b)} \\
\ls & \norm{ g_{x-y} }_{\eta} \absv{ a-b}^{\eta} \ls \norm{ f}_{\zeta+\eta} \absv{ x-y}^{\zeta} \absv{a-b}^{\eta}
\end{align*}

\end{proof}

We also have the following heat kernel estimates - their proof can be found, for example in \cite{MW17}, Proposition 5. \bluepassage{The following proposition is proven there for weighted Besov spaces $B^{\alpha,\mu}_{p,q}$, and is therefore also valid for the special case of $p=q=\infty, \mu = 0$. See equation (2.5) in their work for the definition of weights and accompanying parameters and equation (3.6) for the definition of weighted Besov space} 

\begin{proposition}\label{prop:heatkernel}Let $\alpha \geq \gamma \in \RR, f \in B^{\alpha}_{\infty,\infty}$ and $P_t$ be the heat kernel. Then:

\begin{equation}
\norm{ P_t f }_{B^{\alpha}_{\infty,\infty}} \lesssim t^{-(\alpha-\gamma)/2} \norm{ f}_{B^{\gamma}_{\infty,\infty}} 
\end{equation}

\begin{equation}
\norm{ P_t \nabla^k f}_{B^{\alpha}_{\infty,\infty}} \lesssim t^{-(\alpha - \gamma + k)/2} \norm{f}_{B^{\gamma}_{\infty,\infty}}
\end{equation}

\end{proposition}

\bibliography{main}
\bibliographystyle{alpha}
\end{document}